\documentclass[reqno,10pt]{amsart}
\usepackage{amsmath}
\usepackage{amssymb}
\usepackage{amsfonts}
\usepackage{graphicx}
\usepackage[abs]{overpic}
\usepackage{caption}
\usepackage{subcaption}
\usepackage{wrapfig}
\usepackage{float}
\usepackage{graphicx}
\usepackage{physics}
\usepackage{esint} 
\usepackage{braket} 

\usepackage{tikz}
\usetikzlibrary{decorations.pathreplacing,calc}
\def\centerarc[#1](#2)(#3:#4:#5)
{ \draw[#1] ($(#2)+({#5*cos(#3)},{#5*sin(#3)})$) arc (#3:#4:#5); }

\definecolor{blue_links}{RGB}{13,0,180} 

\usepackage{hyperref}
\hypersetup{
    colorlinks=true, 
    linktoc=all,     
    linkcolor=blue_links,  
    citecolor=blue_links,
    urlcolor=blue_links,
}

\usepackage{mathtools}

\usepackage{xcolor}

\newtheorem{theorem}{Theorem}[section]
\newtheorem{lemma}[theorem]{Lemma}

\newtheorem{remark}[theorem]{Remark}

\newtheorem*{theorem*}{Theorem}

\newcommand{\N}{\mathbb{N}}

\newcommand{\R}{\mathbb{R}}
\newcommand{\C}{\mathbb{C}}
\newcommand{\CC}[1]{\C #1:#1}

\newcommand{\EEE}{\color{black}}

\newcommand{\RRR}{\color{red}}

\newcommand{\dist}{\operatorname{dist}}
\def\eps{\varepsilon}

\def\weakly{\rightharpoonup}

\def\dist{\operatorname{dist}}
\def\Xint#1{\mathchoice
    {\XXint\displaystyle\textstyle{#1}}%
    {\XXint\textstyle\scriptstyle{#1}}%
    {\XXint\scriptstyle\scriptscriptstyle{#1}}%
    {\XXint\scriptscriptstyle\scriptscriptstyle{#1}}%
\!\int}
\def\XXint#1#2#3{{\setbox0=\hbox{$#1{#2#3}{\int}$}
\vcenter{\hbox{$#2#3$}}\kern-.5\wd0}}

\def\dashint{\Xint-}

\usepackage{color}

\numberwithin{equation}{section}

\begin{document} 

\title[Phase separation in a fractured material]{An elliptic approximation for \\ phase separation in a fractured material}
\author[K. Stinson] {Kerrek Stinson} 
\address[Kerrek Stinson]{Department of Mathematics, University of Utah, Salt Lake City, UT 84112, USA}
\email{kerrek.stinson@utah.edu}
\author[S. Wittig]{Solveig Wittig}
\address[Solveig Wittig]{University of Bonn, Endenicher Allee 62, 53115 Bonn,
Germany}

\RRR 
\subjclass[2020]{49J45, 74B99, 74G65, 74N99, 74R10}
\keywords{Phase separation, linear elasticity, fracture, free-boundary, free-discontinuity}

\EEE
\begin{abstract} 
We consider a free-boundary and free-discontinuity energy connecting phase separation and fracture in an elastic material. The energy excludes the contribution of phase boundaries in the cracked region, providing a heuristic approximation of the interfacial energy in the current material configuration. Our primary result shows that the sharp energy may be recovered via $\Gamma$-convergence from a modified Cahn--Hilliard energy coupled with an Ambrosio--Tortorelli-type approximation of the (linear) elastic and fracture energy.
\end{abstract}
\EEE

\maketitle

\section{Introduction}

In many materials, phase separation and damage processes occur simultaneously. For instance, in lithium-ion batteries (LiFePO$_4$), ions will separate into regions of high or low concentration, like oil and water. But to account for an influx of ions, the solid host material (FePO$_4$) must expand. As neighboring regions may have dramatically different concentrations of ions, the induced strain can vary suddenly, leading to large increases in elastic energy. At that point, it may be favorable to simply break the host material to offset the cost in elastic energy. Ultimately, these damage processes lead directly to a decrease in life and performance of the battery \cite{Bazant-Theory2013,Dal2015-Comp,Li_2020}. Overcoming these limitations requires phenomenologically accurate models that capture the interrelated effects of phase separation and damage.

The interaction between phase separation and damage is not unique to batteries and has been studied in a variety of contexts. Much work relies on describing the system's state by phase field energies. 
On a material domain $\Omega\subset \R^d$, prototypical choices are  the Cahn--Hilliard energy for the chemical energy given by
\begin{equation}\label{eqn:CH}
E_{\rm CH} : = \int_{\Omega} \left(\frac{1}{\eps}W(c) + \eps|\nabla c|^2 \right) dx,
\end{equation}
where $c:\Omega\to \R$ is the underlying concentration, $W(s)$ is a two-well function vanishing only at $s=0$ and $s=1$, and $0<\eps\ll 1$ is a parameter describing interface width, and an Ambrosio--Tortorelli-type approximation of the elastic and damage energy given by
\begin{equation}\label{eqn:AT}
E_{\rm AT} : = \int_{\Omega} (z^2+ \delta^2) \CC{(e(u)-ce_0)}\, dx + \int_{\Omega} \left(\frac{1}{\delta}V(z) + \delta|\nabla z|^2 \right) dx,
\end{equation}
depending on the material displacement $u:\Omega \to \R^d$, a fourth-order elastic tensor $\mathbb{C}$, the symmetric gradient $e(u) := (\nabla u + \nabla u^T)/2$, the lattice misfit $e_0\in \R^{d\times d}_{\rm sym}$, a damage variable $z:\Omega\to [0,1]$, a single well function $V(s)$ vanishing at $s=1$, and $0<\delta\ll 1$.
Using the underlying free energy $E_{\rm CH} + E_{\rm AT}$, Heinemann et al. \cite{HK_2013,heinemannKraus14,heinemannKraus_2015,HCRR_2017} considered a variety of models coupling phase separation, elasticity, and damage. In their framework, damage inhibits the transport of a concentration through the mobility; in the simplest setting, the Cahn--Hilliard equation is given by ${\partial_t c = {\rm div}(m(z) \nabla \mu)}$, where $\mu$ is the chemical potential (i.e., first variation of the free energy with respect to the concentration) and the mobility $m:[0,1]\to [0,1]$ is possibly degenerate with $m(0) = 0$. We refer also to \cite{O_Connor_2016} where a similar approach was used to investigate phase separation in batteries numerically.

Instead of using phase field models, one can use interface and free-discontinuity models to investigate phase boundaries and fracture---the connection being that the $\Gamma$-limits of \eqref{eqn:CH}, as $\eps\to 0$, and \eqref{eqn:AT}, as $\delta\to 0$, are a perimeter functional and a  Griffith-type energy, respectively.
A recent work of Bresciani et al. \cite{brescianiFriedrichMoraCorral} studies lower semi-continuity for free-discontinuity energies with both Eulerian (defined on the current configuration) and Lagrangian (defined on the reference configuration) contributions. Roughly summarizing their results in the special case of nonlinear elasticity with a deformation $u:\Omega\to \R^d$, a crack $K\subset \Omega$, and a concentration $c:u(\Omega\setminus K)\to \{0,1\}$, they show that energies penalizing elasticity, fracture, and the length of the phase boundary in the current configuration $u(\Omega\setminus K)$ are lower semi-continuous. We also refer to \cite{henaoMoraCorral2010,silhavy_2011} for related works on cavitation or phase boundaries in elastic domains.

The purpose of this paper is to analyze an elliptic (or diffuse) approximation for a model accounting directly for the interaction of damage and phase separation at the level of the energy. Herein, we tread the line between the two modeling paradigms implicit in the work of Heinemann et al. \cite{HK_2013,heinemannKraus14,heinemannKraus_2015,HCRR_2017} and Bresciani et al. \cite{brescianiFriedrichMoraCorral} as our energy is fully Lagrangian, and thus perhaps simply used for numerics, but simultaneously takes into account Eulerian-type effects by removing ``false'' phase boundaries from the free energy.
 Given a material domain $\Omega \subset \R^d$, a concentration $c:\Omega \to \{0,1\}$, and a material displacement $u\in W^{1,2}_{\rm loc}(\Omega \setminus K)$ breaking the material on the crack-set $K\subset \Omega$, we consider the energy
\begin{equation}\label{eqn:limitEnergyIntro}
E[c,(u,K)] : = \alpha_{\rm surf}\mathcal{H}^{d-1}(\partial \{c=1\}\setminus K) + \int_{\Omega\setminus K} \CC{(e(u) - ce_0)} \, dx + \alpha_{\rm frac}\mathcal{H}^{d-1}(K),
\end{equation}
where $\mathcal{H}^{d-1}$ is the Hausdorff (surface) measure. In the energy \eqref{eqn:limitEnergyIntro}: The first term accounts for the {interfacial} energy due to phase separation in $\Omega\setminus K$ where $\alpha_{\rm surf}>0$ is the interfacial energy density and $\partial \{c=1\}$ denotes the boundary within $\Omega$. As in \eqref{eqn:AT}, the second term is the elastic energy indicating a preferred strain due to the ion-concentration and the lattice misfit. And, the third term captures the energy needed to break the material, with $\alpha_{\rm frac}>0$ being the energy required per unit surface area. As a technical necessity, to ensure that our energy is amenable to tools from the calculus of variations, we will later replace the pair $(u,K)$ by a single function $u\in GSBD^2(\Omega)$ with its jump-set $J_u$ taking the role of~$K.$

\begin{figure}
\centering
\begin{tikzpicture}[x=1cm,y=1cm]
\draw [thick,dotted, fill=blue!10]  (0,-3.0) to[out=180,in=315] (-2,-2)  to[out=135,in=270] (-3,0) to[out=90,in=225] (-2,2.0) to[out=315,in=110] (1,0.8) to[out=200,in=45] (0,0.3) to[out=225,in=80] (-0.5,-0.8) to[out=350,in=90] (0,-3);
\draw [very thick,red] (-1,-1.5) to[out=40,in=260] (-0.5,-0.8) to[out=80,in=225] (0,0.3) to[out=45,in=200] (1,0.8) to[out=20,in=250] (2.0,2.0);
\draw [thick] (3,0) to[out=270,in=45] (2,-2.0) to[out=225,in=0] (0,-3.0) to[out=180,in=315] (-2,-2)  to[out=135,in=270] (-3,0) to[out=90,in=225] (-2,2.0) to[out=45,in=180] (0,3.0) to[out=0,in=135] (2,2) to[out=315,in=90] (3,0) ;
\node[] at (-1.7,-2.75) {$\Omega$};
\node[] at (-1.75,0.5) {$c=1$};
\node[] at (1.75,-0.5) {$c=0$};
\node[above] at (1.4,1.2) {$K$};
\begin{scope}[xshift=8cm]
\draw [thick,dotted, fill=blue!10]  (0,-3.0) to[out=180,in=315] (-2,-2)  to[out=135,in=270] (-3,0) to[out=90,in=225] (-2.25,2.0) to[out=315,in=110] (0.9,0.9) to[out=200,in=45] (-0.075,0.375) to[out=225,in=80] (-0.55,-0.75) to[out=260,in=40] (-1,-1.5) to[out=40,in=260] (-0.45,-0.85) to[out=350,in=90] (0,-3);
\draw [thick] (-1,-1.5) to[out=40,in=260] (-0.55,-0.75) to[out=80,in=225] (-0.075,0.375) to[out=45,in=200] (0.9,0.9) to[out=20,in=250] (1.75,2.25);
\draw [thick] (-1,-1.5) to[out=40,in=260] (-0.45,-0.85) to[out=80,in=225] (0.075,0.225) to[out=45,in=200] (1.1,0.7) to[out=20,in=250] (2.25,1.75);
\draw [thick] (2.25,1.75) to[out=315,in=90] (3,-0.3) to[out=270,in=45] (2,-2.25) to[out=225,in=0] (0,-3.0) to[out=180,in=315] (-2,-2)  to[out=135,in=270] (-3,0) to[out=90,in=225] (-2.25,2.0) to[out=45,in=180] (-0.3,3.0) to[out=0,in=135] (1.75,2.25)  ;
\node[] at (-2.75,-2.75) {$({\rm id}+u)(\Omega\setminus K)$};
\end{scope}
\draw [->] plot [smooth] coordinates {(3.0,1.2) (4.0,1.4) (5.0,1.2)};
\node[above] at (4.0,1.5) {${\rm id} + u$};
\end{tikzpicture}
\caption{The graphic illustrates the idea behind the first term in energy \eqref{eqn:limitEnergyIntro}. The transition $c$ makes from $0$ to $1$ on the crack $K$ is a ``false'' phase boundary only seen in the reference configuration, on the left side. In the current configuration, on the right side, this transition takes place on the boundary of the new domain.}
\label{fig:splitCrack}
\end{figure}
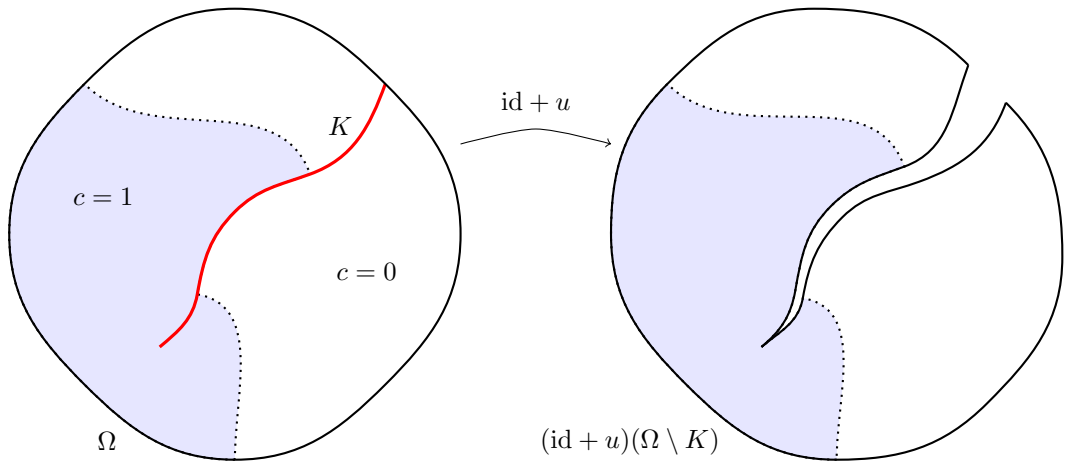

 The energy \eqref{eqn:limitEnergyIntro} is heuristically motivated by Figure \ref{fig:splitCrack}. In particular, one sees that when the crack opens up the material body, phase boundaries can no longer ``see" eachother across the crack in the current configuration, with the intuition behind the first term in \eqref{eqn:limitEnergyIntro} being that it approximates the phase boundary energy in the current configuration. Our approach is entirely Lagrangian, and while a truly Eulerian--Lagrangian energy might be needed to completely account for the effects of fracture on phase boundaries, such an approach brings a plethora of technical challenges even for lower semi-continuity (see \cite{brescianiFriedrichMoraCorral,silhavy_2011}), which is only half of the battle in approximation.

Recalling the notation in \eqref{eqn:CH} and \eqref{eqn:AT}, we introduce the phase-field energies approximating \eqref{eqn:limitEnergyIntro} given by
\begin{equation}\label{eqn:energyEpsIntro}
\begin{aligned}
E_{\eps,\delta}[c,u,z] : = & \int_{\Omega} \phi_\delta(z) \left(\frac{1}{\eps}W(c) + \eps |\nabla c|^2\right) dx \\
& \qquad  + \int_{\Omega} (z^2 + \delta^2)\CC{(e(u) - ce_0)}\, dx + \int_{\Omega} \left( \frac{1}{\delta}V(z) + \delta|\nabla z|^2 \right) dx,
\end{aligned}
\end{equation}
for $(c,u,z) \in H^1(\Omega)\times H^{1}(\Omega;\R^d) \times H^1(\Omega;[0,1])$, where $z$ is a damage variable providing a diffuse approximation of the crack-set with $\{z\neq 1 \}\approx K$. The precise assumptions on the functions appearing in \eqref{eqn:energyEpsIntro} are laid out in Section \ref{sec:mainResults}, but for now it suffices to say that $\phi_\delta(z)$ behaves like $|z| + \delta^2$. This modifies the simple addition of $E_{\rm CH} + E_{\rm AT}$ by removing the phase boundary energy in damaged regions.

Our primary result rigorously connects the energy \eqref{eqn:energyEpsIntro} to \eqref{eqn:limitEnergyIntro} via $\Gamma$-convergence (see, e.g., \cite{Brad02,DalMasoBook}), which we formally state below and refer to Theorem \ref{thm:main} for the complete statement. We emphasize that $\Gamma$-convergence is a natural notion for describing the convergence of variational problems that has been used within materials science and outside of it (see, e.g., \cite{bungertStinson2024,conti-SO2rigid,Modica87,stinsonLiBatteryGamma}).

\begin{theorem*} Suppose that $\eps \to 0$ and $\delta = \delta(\eps) \to 0$ with $\eps/\delta\to 0$. Then under the topology induced by convergence in measure for the triplet $(c,u,z)$, the diffuse energy \eqref{eqn:energyEpsIntro} $\Gamma$-converges to the sharp interface energy \eqref{eqn:limitEnergyIntro}.
\end{theorem*}

This result is also accompanied by a compactness theorem (see Theorem \ref{thm:compactness} below) showing that finite energy sequences $(c_{\eps},u_\eps,z_\eps)$ converge in measure, up to a subsequence, to a triplet $(c,u,1)$. To deal with a loss of control resulting from frame-indifference and no lower order control on $u$, the convergence holds for a piecewise-affine modification of the displacement $u$, an approach {also} adopted in \cite{FriedrichSolombrino18,friedrich2020,chambolleCrismale-Hetero}. In contrast to what is often done in the homogeneous setting of the Griffith energy (see, e.g., \cite{AlmiTasso_2023,chambolleCrismale18}), we keep track of domain pieces which ``travel off to infinity" to provide a robust $\Gamma$-convergence result that can, for instance, incorporate mass constraints on the concentration (see Remark \ref{rmk:massConstraint}). Fundamentally, our $\Gamma$-convergence result and compactness theorem show that minimizers of $E_{\eps,\delta}$ converge to minimizers of $E$, even with constraints.

We briefly comment on some points of mathematical interest in the $\Gamma$-convergence theorem.
Approximation via $\Gamma$-convergence stands on two pillars: a $\liminf$ and a $\limsup$ inequality. In the proof of the $\liminf$ inequality, {we must ensure} that the ``energetic mixing'' of phase energy and damage energy does not prevent lower semi-continuity. Additionally, we show how slicing can be done with the convergence of piecewise-affine modified displacements. The proof of the $\limsup$ inequality requires two technical lemmas; the first an approximation for sets of finite perimeter, under which the energy $\mathcal{H}^{d-1}(\partial\{c=1\}\setminus K)$ is stable, and the second a means to exclude the phase boundary energy in damaged regions for the diffuse recovery sequence.

\textbf{Organization.} In the next section, we introduce notation and discuss some mathematical necessities, such as $BV$ and $GSBD$ function spaces and sets of finite perimeter. In Section \ref{sec:mainResults}, we state our results in technical detail. In Section~\ref{sec:compactness}, we prove the associated compactness theorem. The proof of the $\Gamma$-limit is broken into Sections~\ref{sec:liminf} and~\ref{sec:limsup} for the $\liminf$ and $\limsup$ inequality, respectively.

\section{Notation and mathematical preliminaries}\label{sec:mathPrelim} 
We use $C>0$ to represent a generic constant, potentially changing from line to line. To emphasize dependence of the constant on a parameter $\rho$, we will write $C(\rho)$ or $C_\rho$. We use $\mathcal{L}^d$ for the $d$-dimensional Lebesgue measure and $\mathcal{H}^{d-1}$ for the $(d-1)$-dimensional Hausdorff measure. We will occasionally use the phrase \emph{infinitesimal rigid motion} to mean an affine map $a(x) = Ax + b$, where $A\in \R^{d\times d}_{\rm skew}$ and $b \in \R^d$. We use $\R^{d\times d}_{\rm sym}$ and $\R^{d\times d}_{\rm skew}$ for symmetric and skew-symmetric matrices, respectively. When possible, we will denote the derivative of a function $v:\R\to \R$ by $v'$ but occasionally for notational clarity use $\frac{d}{dt}v.$

Throughout, we will use fine properties of $BV(\Omega)$ or $GSBD^2(\Omega)$ functions for an open and bounded set $\Omega\subset \R^d$. 
We recall that $v\in BV(\Omega)$ if and only if $v\in L^1(\Omega)$ and its distributional gradient $D v$ is a Radon measure with finite total variation (in $\Omega$). For a function $v$ within $BV(\Omega)$, we use $\nabla v$ to denote the absolutely continuous part of the gradient. The jump-set of $v$ is denoted by $J_v$, which consists of the points in $\Omega$ for which $v$ is not continuous in a measure-theoretic sense. At $\mathcal{H}^{d-1}$-almost every $x\in J_v$, there are approximate one-sided limits which we denote by $v^+(x)$ and $v^-(x)$, with $v^+(x)>v^-(x)$ (note the $\pm$ has nothing to do with a chosen orientation for the normal to $J_v$). For $p\in [1,\infty)$, we say that $v\in SBV^p(\Omega)$ if $\nabla v \in L^p(\Omega;\R^d)$, $\mathcal{H}^{d-1}(J_v)<\infty$, and $Dv$ has no Cantor part.

We say that $A\subset \Omega$ is a set of finite perimeter (relative to $\Omega$) if $\chi_A\in BV(\Omega)$, where $\chi_A$ is the characteristic function of the set taking the value $1$ on $A$ and $0$ otherwise. We denote the reduced boundary (in $\Omega$) of a set of finite perimeter $A$ by $\partial^* A$. A Caccioppoli partition of $\Omega$ is a pairwise disjoint collection of sets of finite perimeter $\mathcal{P}:=\{P_i\}_{i\in \N}$ contained in and covering $\Omega$ (up to a null-set) with bounded cumulative perimeter, i.e., $$\mathcal{L}^d(\Omega \setminus (\cup_i P_i)) = 0 \quad \text{ and }\quad \sum_i \mathcal{H}^{d-1}(\partial^* P_i)<\infty.$$ We will occasionally use the shorthand $\partial^* \mathcal{P}: = \cup_i \partial^*P_i.$ See \cite{AmbrosioFuscoPallara} for more information on $BV$ functions, sets of finite perimeter, and Caccioppoli partitions.

The function space $GSBD^2(\Omega)$ of functions $u:\Omega \to \R^d$ is essentially the class of functions for which the Griffith energy is finite, i.e.,
$$\int_{\Omega} |e(u)|^2\, dx  + \mathcal{H}^{d-1}(J_u)<\infty,$$
where in parallel to $BV$ functions, $e(u)$ is the absolutely continuous part of the symmetric gradient and $J_u$ is the set of discontinuity points for $u$---however, these must be {carefully defined via slicing since there is no control on the full gradient}. We will rely on slicing for the proof of the $\liminf$ inequality for $d>1$ and introduce the necessary machinery in Remark \ref{rmk:slicing}. As we will otherwise be able to avoid technical complications, we refer the interested reader to the canonical paper \cite{dalMasoGSBD} for further information on $GSBD$ functions.

\section{Model and main results}\label{sec:mainResults}

Throughout, we let $\Omega \subset \R^{d}$ be a bounded open set with Lipschitz boundary and   $\eps_n \to 0$, $\delta_n \to 0$ be given sequences; we always drop the subscript $n$ and implicitly take $\delta = \delta(\eps)$, while keeping in mind these refer to sequences.
Our diffuse energy is given by 
\begin{equation}\label{eqn:energyEps}
\begin{aligned}
E_{\eps,\delta}[c,u,z] : = & \int_{\Omega} \phi_\delta(z) \left(\frac{1}{\eps}W(c) + \eps |\nabla c|^2\right) dx \\
& \qquad  + \int_{\Omega} (z^2 + \delta^2)\CC{(e(u) - ce_0)}\, dx + \int_{\Omega} \left( \frac{1}{\delta}V(z) + \delta|\nabla z|^2 \right) dx,
\end{aligned}
\end{equation}
if $(c,u,z) \in H^1(\Omega) \times H^1(\Omega;\R^d) \times H^1(\Omega;[0,1])$ and $+\infty$ otherwise. Here $e_0 \in \R^{d\times d}_{\rm sym}$ is the lattice misfit matrix, encoding the difference in solid structure between the saturated ($c=1$) and unsaturated ($c=0$) phases, and $\C:\R^{d\times d}\to \R^{d\times d}_{\rm sym}$ is a positive-definite fourth-order tensor of material constants with $$ C |\xi|^2 \geq \CC{\xi} \geq \frac{1}{C}\left|\xi\right|^2 \quad \text{ for all }\xi \in \R^{d\times d}_{\rm sym}$$ for a positive constant $C>0.$
We constrain the functions in the energy as follows: We assume that the double-well function $W$ satisfies
\begin{equation}\label{ass:W0}
\text{$W: \R\to [0,\infty) $ is continuous and $W(s) = 0$ if and only if $s = 0,1$,}
\end{equation}
and is $L^1$-coercive as given by
\begin{equation}\label{ass:W1}
\text{there exists $C>0$ such that $W(s) \ge  \frac{|s|}{C}$ for all $|s|\ge C$}.
\end{equation}
We assume that the single-well function $V$ satisfies $V'\in C^1([0,1])$, $V'|_{[0,1]}\leq 0$, 
\begin{equation}\label{ass:V0}
\text{$V: [0,1] \to [0,\infty) $ is continuous and $V(s) = 0$ if and only if $s = 1$.}
\end{equation}
We impose that
\begin{equation}\label{ass:phase<frac}
\int_0^1 \sqrt{W(s)}\, ds \le 2\int_0^1 \sqrt{V(s)}\, ds
\end{equation}
and
 take $\phi\in C([0,1])$ to be a continuous, non-decreasing function with $\phi(0)=0$ and $\phi(1)=1$, such that $\phi$ is strictly positive on $(0,1]$ with
 \begin{equation}\label{ass:liminf}
  \int_0^1 \sqrt{W(s)}\, ds \le 2\int_{m}^1 \sqrt{V(s)}\, ds +  \phi(m) \int_0^1 \sqrt{W(s)}\, ds \quad \text{ for any } m \in [0,1].
 \end{equation}
 Given \eqref{ass:phase<frac}, one possible choice satisfying \eqref{ass:liminf} is  $\phi(m): = \int_{0}^m\sqrt{V(s)}\, ds /(\int_{0}^1\sqrt{V(s)}\, ds)$.
 With this, we define $\phi_\delta : = \phi + C_\delta \in C([0,1])$, where $\{C_\delta\}_{\delta} \subset (0,\infty)$ is any positive vanishing sequence, i.e., $C_\delta \to 0$ as $\delta \to 0.$ 
We remark that the condition \eqref{ass:liminf} prevents the ``energetic mixing'' of cracks and phase boundaries in the diffuse energy, with its role becoming apparent in the $\liminf$ inequality for $\Gamma$-convergence.

We introduce the limiting energy 
\begin{equation}\label{eqn:limitEnergy}
E[c,u] : = \alpha_{\rm surf}\mathcal{H}^{d-1}(\partial^* \{c=1\}\setminus J_u) + \int_{\Omega\setminus J_u} \CC{(e(u) - ce_0)} \, dx + \alpha_{\rm frac}\mathcal{H}^{d-1}(J_u),
\end{equation}
 if $(c,u)\in BV(\Omega;\{0,1\})\times GSBD^2(\Omega)$ (implicitly $z=1$ almost everywhere) and $+\infty$ otherwise, where 
\begin{equation}\label{eqn:energyDensities}
\alpha_{\rm surf} : = 2\int_0^1 \sqrt{W(s)}\, ds \quad \text{ and } \quad \alpha_{\rm frac} : = 4\int_0^1 \sqrt{V(s)}\, ds.
\end{equation}
At a practical level, one can now see that \eqref{ass:phase<frac} ensures that $\alpha_{\rm surf}\leq \alpha_{\rm frac}$, a necessity for \eqref{eqn:limitEnergy} to be lower semi-continuous. 

Before stating our primary result, we motivate the topology considered therein with the following compactness result for finite energy sequences.
\begin{theorem}\label{thm:compactness}
Let $\Omega \subset \R^{d}$ be an open, bounded set with Lipschitz boundary and suppose  $\eps \to 0$, $\delta = \delta(\eps) \to 0$  with $\eps/\delta \to 0$. Further suppose assumptions \eqref{ass:W0}-\eqref{ass:liminf} hold. If 
$$\limsup_{\eps\to 0}E_{\eps,\delta}(c_\eps,u_\eps,z_\eps) =: C_0 <\infty,$$
then up to a subsequence in $\eps$, $c_\eps \to c \in BV(\Omega ;\{0,1\})$ in measure and $z_\eps\to 1$ in measure. Further, there exists a Caccioppoli partition of $\Omega$ given by $\mathcal{P} = \{P_i\}_{i\in \N}$ and infinitesimal rigid motions (associated to $u_\eps$ on each $P_i$) $a_{\eps,i} (x) : = A_{\eps,i}x + b_{\eps,i}$, with $A_{\eps,i}\in \R^{d\times d}_{\rm skew}$ and $b_{\eps,i}\in \R^d$, such that
\begin{equation}\label{eqn:displacementConvergence}
\left(u_\eps - \sum_i a_{\eps,i} \chi_{P_i}\right) \to u \in GSBD^2(\Omega) \quad \text{ in measure.} 
\end{equation} Further the partition has finite length in the sense that
\begin{equation}\label{eqn:perimControl}
\sum_{i} \mathcal{H}^{d-1}(\partial^*P_i) \leq C(C_0) 
\end{equation}
and the rigid motions separate in the sense that
\begin{equation}\label{eqn:separatingFrames}
|a_{\eps,i}(x) - a_{\eps,j}(x)| \to \infty \quad \text{ as } \eps\to 0 \text{ for almost every }x\in \Omega \text{ and any } i\neq j.
\end{equation}
\end{theorem}

We emphasize that the convergence of $u_\eps$ can only hold up to affine modifications. Related approaches have been developed in a variety of works \cite{friedrich2018-piecewiseKorn,FriedrichSolombrino18,friedrich2020,kholmatovPiovano}, and our result is primarily a consequence of the results in Chambolle and Crismale in \cite{chambolleCrismale-Hetero,chambolleCrismaleEquilibrium}. For the standard Griffith energy, typically one can avoid worrying about the pieces that travel off to infinity by simply setting $u = 0$ there. In our setting however, there is a heterogeneity due to the concentration $c$ and the finer information coming from the partition $\mathcal{P}$ is necessary for a complete view of the limiting energy---in particular, for convergence of minimizers under constraints. Finally, we remark that \eqref{eqn:perimControl} provides a suboptimal control on the length of the boundary, as one would like to ensure that $\mathcal{H}^{d-1}(\cup \partial^* P_i)$ is controlled by the diffuse crack energy, since points in the boundary of the partition should effectively belong to the crack due to \eqref{eqn:separatingFrames}. This intuition is verified by the $\liminf$ inequality of Theorem \ref{thm:main} below. For convenience, given a Caccioppoli partition $\mathcal{P} = \{P_i\}_{i\in \N}$ of $\Omega$, we use the notation $\partial^* \mathcal{P} : = \cup_i \partial^* P_i$.

Our primary result is the following theorem.
\begin{theorem}\label{thm:main}
Let $\Omega \subset \R^{d}$ be a bounded open set with Lipschitz boundary and suppose  $\eps \to 0$, $\delta = \delta(\eps) \to 0$  with $\eps/\delta \to 0$. Further suppose assumptions \eqref{ass:W0}-\eqref{ass:liminf} hold.  Then the energies $E_{\eps,\delta}$ defined in \eqref{eqn:energyEps} $\Gamma$-converge to $E$ defined in \eqref{eqn:limitEnergy} in the topology defined by convergence in measure for the triplet $(c,u,z)$. Precisely, the following holds:
\begin{enumerate}
\item \emph{($\liminf$ inequality).} For any sequence $(c_\eps, u_\eps, z_\eps)$ converging to $(c,u,1)$ in the sense of Theorem~\ref{thm:compactness}, we have
\begin{equation}\label{eqn:liminfInequality}
\begin{aligned}
\alpha_{\rm surf}\mathcal{H}^{d-1}(\partial^* \{c=1\}\setminus {(J_u\cup \partial^* \mathcal{P})}) + \int_{\Omega\setminus J_u} \CC{(e(u) - ce_0)} \, dx \, + \, &   \alpha_{\rm frac}\mathcal{H}^{d-1}(J_u\cup \partial^* \mathcal{P}) \\
&\leq \liminf_{\eps \to 0} E_{\eps,\delta} [c_\eps,u_\eps, z_\eps].
\end{aligned}
\end{equation}
\item \emph{($\limsup$ inequality).} For any ${(c,u)\in BV(\Omega;\{0,1\})\times GSBD^2(\Omega)}$, there is a sequence $(c_\eps, u_\eps, z_\eps)\in H^1(\Omega)\times H^1(\Omega; \R^d)\times H^1(\Omega;[0,1])$ converging to $(c,u,1)$ in measure such that 
\begin{equation}\nonumber
E[c,u] = \lim_{\eps \to 0} E_{\eps,\delta} [c_\eps,u_\eps, z_\eps].
\end{equation}
\end{enumerate}
\end{theorem}

As is typical, Theorems \ref{thm:compactness} and \ref{thm:main} show that a minimizer of $E$ can be found as the (modified) limit of minimizers $(c_\eps,u_\eps,z_\eps)$ for $E_{\eps,\delta}$.
Before turning to the proofs of the above results, we comment on some minor perturbations of their statements, which still hold.

\begin{remark}[Intermediate phase boundary energy]
It is reasonable to think that in place of the energy $E$ in \eqref{eqn:limitEnergy}, one would like to recover the energy
$$E[c,u] + \theta \alpha_{\rm surf}\mathcal{H}^{d-1}(\partial^*\{c=1\} \cap J_u),$$
for $\theta\in (0,1).$ Morally, the above energetic penalization says that energy due to phase boundaries is diminished by damage but not prevented. This energy is the $\Gamma$-limit of $E_{\eps,\delta}$ if one simply modifies the assumption on $\phi$: In addition to being continuous, non-decreasing, and satisfying \eqref{ass:liminf}, $\phi$ must satisfy $\phi(0) = \theta$ and $\phi(1) = 1$.
\end{remark}

\begin{remark}[Elastic energy]
The above theorems hold if the elastic portion of the diffuse energy $E_{\eps,\delta}$, i.e., ${\int_\Omega (z^2+\delta^2)\CC{(e(u) - ce_0)}\, dx},$ is replaced by $$\int_\Omega (\psi(z)+\eta)\CC{(e(u) - ce_0)}\, dx,$$ where $\psi:[0,1]\to [0,1]$ is a non-decreasing continuous function with $\psi(s) = 0$ only if $s=0$ and $\psi(1) = 1$, and the parameter $\eta = \eta(\delta)>0$ satisfies $\eta/\delta \to 0$ as $\delta\to 0.$
\end{remark}

\begin{remark}[Role of the damage variable]
As is common in the literature, we have restricted the damage variable to take values in the interval $[0,1]$, with $z=0$ representing complete fracture and $z=1$ representing undamaged material. Due to the elliptic nature of the approximation (and thereby access to maximum principles) restricting $z$ to take values in $[0,1]$ can almost be done for free. More generally, if one takes $V: \R\to [0,\infty) $ to be continuous, satisfy \eqref{ass:V0}, and be coercive with
\begin{equation}\nonumber 
\text{there exists $C>0$ such that $V(s) \ge  \frac{|s|}{C}$ for all $|s|\ge C$},
\end{equation}
then Theorems \ref{thm:compactness} and \ref{thm:main} still hold.
\end{remark}

\begin{remark}[Mass constraints for the concentration]\label{rmk:massConstraint}
Often times in applications, such as for the Cahn--Hilliard equation, the total concentration in $\Omega$ is fixed over time, i.e., $\dashint_\Omega c \, dx = \mu_0$. Similarly, with a logarithmic potential in place of $W$ (see, e.g., \cite{garcke-CHlogPot,stinsonLiBatteryExpBC}), the concentration is constrained to take values between $[0,1]$. Given $\mu_0\in [0,1]$, Theorem \ref{thm:compactness} and \ref{thm:main} hold under the additional condition that $E_{\eps,\delta}(c_\eps,u_\eps,z_\eps) = +\infty$ and $E[c,u] =+\infty$ if the mass constraint fails, i.e., if $\dashint_\Omega c_\eps \, dx \neq  \mu_0$ or $\dashint_\Omega c \, dx \neq  \mu_0$, or the $L^\infty$-constraint fails, i.e., $c_\eps$ is not in $[0,1]$ almost everywhere. The $\liminf$ inequality for these modified energies follows from the existing proof, since we keep track of the convergence of $c_\eps$ and $u_\eps$ in all of $\Omega$. The constraint that $c_\eps$ take values in $[0,1]$ is only so that the condition $\dashint_\Omega c_\eps \, dx = \mu_0$ is stable as $c_\eps$ converges in measure to $c$; as such, any equi-integrability of the sequence $c_\eps$ will suffice. The proof of the $\limsup$ inequality follows from the one herein and modified as in \cite[Theorem 1.2]{stinsonLiBatteryGamma}, with the $L^\infty$-constraint on $c_\eps$ only used to ensure that convergence in measure then implies $\dashint_\Omega c_\eps \, dx \to \dashint_\Omega c \, dx.$
\end{remark}

 We now turn to proving the compactness Theorem \ref{thm:compactness}. The proof of the $\Gamma$-limit in Theorem \ref{thm:main} is split into Section~\ref{sec:liminf} for the $\liminf$ inequality and Section~\ref{sec:limsup} for the $\limsup$ inequality.

\section{Compactness}\label{sec:compactness}

We first note the well-known trick connecting geodesic distances to the Cahn--Hilliard energy.
We use $f$ as a placeholder for $W$ or $V$ and define $d_f: \R \to \R$ 
by $$d_f(t):=2\int_0^t \sqrt{\min\{f(s), M\}} \, ds$$ where $M:=\|f\|_{L^\infty([0,1])}$.
Then $d_f $ is Lipschitz continuous with derivative $ d_f'(t)=2\sqrt{\min\{f(t), M\}}$ and, as it is strictly increasing, has continuous inverse $d_f^{-1}:\R \to \R$.
For $w \in H^1(\Omega)$ and $\eps >0$, Young's inequality implies for any measurable $A\subset \Omega$ that
\begin{equation}\label{eqn:MMtrick}
    |D [d_f\circ w]|(A)
    = \int_A| 2\sqrt{\min\{f(w), M\}}\nabla w| \, dx
    \le \int_A \left(\frac{1}{\eps }f(w) + \eps| \nabla w|^2\right)  dx
\end{equation}

\begin{proof}[Proof of Theorem \ref{thm:compactness}]
The strategy of the proof is as follows: Convergence of $z_\eps$ will follow from existing results. We roughly take the sets $A_\eps \approx \{|c_\eps|< 3\}$ and $B_\eps \approx \{|z_\eps|>1/2\}$ such that both of these sets have finite perimeter controlled by $C_0.$ By virtue of $W$ and $V$, $A_\eps \cap B_\eps \to \Omega$. Then it is possible to show that $c_\eps\chi_{A_\eps \cap B_\eps}$ and $u_\eps\chi_{A_\eps \cap B_\eps} $ converge in measure as desired by applying $BV$ and $GSBD^2$ compactness, respectively. We note that whenever necessary we take a subsequence in $\eps$ without comment.

\emph{Step 1 (Convergence of $z_\eps$).} Since $\int_\Omega \left(\frac{1}{\delta}V(z_\eps)+ \delta |\nabla z_\eps|^2\right) dx
\le C_0$, the standard compactness result for the Cahn--Hilliard (or Modica--Mortola) functional \cite{ModicaMortola,Modica87} yields that
 $z_\eps \to 1$ in $L^1(\Omega)$.

\emph{Step 2 (Selection of sub/super level-sets).}
In this step, we find large sets with finite perimeter inside which the functions $c_{\eps}$ and $z_{\eps}$ are bounded. Each set is found by a variation of the same argument.

Applying \eqref{eqn:MMtrick} with $\eps$ replaced by $\delta$ {and $f$ by $V$}, we have 
\begin{align}\nonumber
 |D [d_{ V}\circ z_{\eps}]|(\Omega)\leq  C_0. 
\end{align}
Thus from the coarea formula \cite{AmbrosioFuscoPallara} we have
\begin{align}\nonumber
    \dashint_{d_{ V}(\frac{1}{4})}^{d_{ V}(\frac{3}{4})} \mathcal{H}^{d-1}(\partial^*\{d_{ V} \circ z_{\eps} > t\})\, dt
    \le \frac{C_0}{d_{ V}(\frac{3}{4})-d_{ V}(\frac{1}{4})},
\end{align}
and  therefore  for each $\eps>0$ we can find some
 $\tilde t_{\eps}\in (d_{V}(\tfrac{1}{4}),d_{ V}(\tfrac{3}{4}))$ such that 
$\mathcal{H}^{d-1}(\partial^*\{d_{ V} \circ z_{\eps} > \tilde t_{\eps}\}) \le C$ (where $C>0$ does not depend on $\eps$).
Consequently, defining $t_{\eps} := d_{V}^{-1}(\tilde t_{\eps}) $ and $$B_\eps := \{z_{\eps} > t_{\eps}\} =\{d_{V} \circ z_{\eps} > \tilde t_{\eps}\} ,$$ we have $\mathcal{H}^{d-1}(\partial^*B_\eps) \le C$
and by monotonicity of $d_{ V}^{-1}$ it holds that $t_\eps \in (\frac{1}{4}, \frac{3}{4}) $.
Since $z_{\eps}\to 1$ in $L^1(\Omega)$ by Step 1, dominated convergence implies $\chi_{B_\eps}\to \chi_{\Omega} $ in $L^1(\Omega)$.

We continue in a similar fashion by
finding subsets $A_\eps \subset B_\eps \subset \Omega$ with finite perimeter
 inside which $c_{\eps}$ is bounded.
Applying \eqref{eqn:MMtrick} and using the monotonicity of $\phi_{\delta}$, we compute 
\begin{equation}\nonumber
    \limsup_{\eps \to 0}|D [d_{W}\circ c_{\eps}]|  ( B_\eps)
    \le \limsup_{\eps \to 0} \frac{1}{\phi_{\delta} (\frac{1}{4})} \int_{B_\eps} \phi_{\delta} (z_{\eps}) \left(\frac{1}{\eps}W(c_{\eps})+\eps|\nabla c_{\eps}|^2\right) dx
    \le  \frac{C_0}{\phi (\frac{1}{4})} <\infty.
\end{equation}
Thus we may assume $|D [d_{W}\circ c_{\eps}]|  ( B_\eps) \le C$  for all $\eps$ for some $C>0$.
Arguing as before
   \begin{align}
    \dashint_{d_{W}(2)}^{d_{W}(3)} \mathcal{H}^{d-1}(\partial^*\{d_{W} \circ c_{\eps} > t\}\cap B_\eps)\, dt
    \le   \frac{C}{d_{W}(3)-d_{W}(2)},
\end{align}
so there must exist some  $\tilde s_{\eps} \in (d_{W}(2),d_{W}(3))$ such that 
$ \mathcal{H}^{d-1}(\partial^*\{d_{W} \circ c_{\eps} > 
\tilde s_{\eps} \}\cap B_\eps) \le C$.
Setting $s_{\eps} := d_{W}^{-1}(\tilde s_{\eps}) \in (2,3) $ 
we have $ \mathcal{H}^{d-1}(\partial^*\{ c_{\eps} > s_{\eps} \}\cap  B_\eps) \le C$.
In the same way we can also find a lower bound $s_\eps^{-} \in (-3,-2)$ such that $ \mathcal{H}^{d-1}(\partial^*\{ c_{\eps} {>}  s_\eps^{-} \}\cap B_\eps) \le  C$ for some new constant $C>0$.
Defining
$$A_\eps := \{s_\eps^{-} <c_{\eps} {\leq} s_\eps\} \cap B_\eps ,$$ we clearly have $\mathcal{H}^{d-1}(\partial^*A_\eps) \le C$ and $\{ |c_\eps| \le 2\} \cap B_\eps  \subset A_\eps \subset \{ |c_\eps| \le 3\}$.
From this last relation and the properties of $W$ in \eqref{ass:W0} and \eqref{ass:W1}, we find that
$$\mathcal{L}^{d}(\{|c_\eps|>2\}\cap B_\eps) \leq \frac{1}{C\phi_\delta(1/4)}\int_{\Omega} \phi_\delta(z_\eps)W(c_\eps) \, dx\to 0 $$
for some constant $C = C(W)>0$, so that $\chi_{A_\eps} \to \chi_{\Omega}$ in $L^1(\Omega)$.

\emph{Step 3 (Convergence of $c_\eps$).}
It immediately follows that $(d_{W}\circ  c_\eps)\chi_{A_\eps}$ is bounded in $L^1(\Omega)$, and further from the product rule for $BV$ functions,
$$|D[(d_{W}\circ  c_\eps)\chi_{A_\eps}]|(\Omega)\leq C + {6}\sqrt{M} \mathcal{H}^{d-1}(\partial^*A_\eps) ,$$
where we have picked up the factor ${6}\sqrt{M}$ as this is the maximum possible jump between $0$ and $d_{W}\circ  c_\eps$ on the boundary of $A_\eps.$  
By $BV$-compactness \cite{AmbrosioFuscoPallara}, we have that $(d_{W}\circ c_\eps)\chi_{A_\eps} \to w \in BV(\Omega)$ in $L^1(\Omega),$ and defining $c : = d_W^{-1}\circ w$, we have $ c_\eps \chi_{A_\eps}\to c$ in $L^1(\Omega)$. 
Since $\chi_{A_\eps} \to \chi_\Omega$, we recover that $c_\eps$ converges in measure to $c$. Given the singular perturbation of $W$ and that $\chi_{B_\eps}\to \chi_\Omega$, it is direct to conclude that $c$ belongs to $BV(\Omega;\{0,1\}).$

\emph{Step 4 (Convergence of $u_\eps$).}
We prove that the functions $\bar u_\eps : = u_\eps \chi_{A_\eps}$ converge as claimed in the theorem statement, from which the convergence for $u_\eps$ will follow as $\chi_{A_\eps} \to \chi_\Omega$ in $L^1(\Omega)$. Note that using the coercivity of $\C$, we have
\begin{align*}
&\int_{\Omega} |e(\bar u_\eps)|^2\, dx + \alpha_{\rm frac}\mathcal{H}^{d-1}(J_{\bar u_\eps}) \\
&\leq 2\int_{B_\eps} |e( u_\eps) -c_\eps e_0|^2 \, dx + 2\int_{A_\eps} |c_\eps e_0|^2 \, dx + \alpha_{\rm frac}\mathcal{H}^{d-1}(\partial^*A_\eps) \leq C(C_0) + 18|e_0|^2\mathcal{L}^d(\Omega) + C,
\end{align*}
and
we may directly apply the compactness result \cite[Theorem 1.1]{chambolleCrismaleEquilibrium} to conclude the proof.
\end{proof}

\section{The $\liminf$ inequality}\label{sec:liminf}
We now turn to the $\liminf$ inequality of Theorem \ref{thm:main}. The argument revolves around slicing, reducing much of the effort to verification of the $\liminf$ inequality in $1$D. In $1$D, since the symmetrized gradient is simply the gradient in $1$D, the main challenge is to pick up the correct {surface} energy coming from the jumps of $c$ and $u$ under the convergence of modified sequences found via Theorem \ref{thm:compactness}; here, we build off of a strategy used in \cite{ambrosio-tortorelli-1990} (see also Braides' book \cite[Chapter 4]{braidesFreeDiscont}).
At the end, a little work is required to show that slicing does not lose too much energy and can fully recover the energy of the symmetric gradient when $d>1$; for this, we extend the approach used by Iurlano in \cite{iurlano_density} (see also \cite{chambolleCrismale_approx}) to the heterogeneous $GSBD$ setting.

\begin{proof}[Proof of the $\liminf$ inequality of Theorem \ref{thm:main} in $1$D] 
We begin by noting that when $d=1$, $GSBD^2(\Omega) = GSBV^2(\Omega;\R)$ and the affine modifications of $u_\eps$ are given by piecewise constant shifts, so that
\begin{equation}\label{eqn:measConverge1D}
\left(u_\eps - \sum_i a_{\eps,i}\chi_{P_i}\right)  \to u \in GSBV^2(\Omega;\R) \quad  \text{ in measure},
\end{equation}
with $a_{\eps,i} \in \R$ and 
\begin{equation}\label{eqn:sepFrame2}
|a_{\eps,i} - a_{\eps,j}|\to \infty
\end{equation} for all $i\neq j$. Naturally, in the energy \eqref{eqn:energyEps}, we may assume that $\CC{(e(u)-ce_0)}$ is replaced by $|u' - ce_0|^2$ with $e_0\in \R.$ Without loss of generality, we assume 
\begin{equation}\nonumber
\lim_{\eps \to 0}E_{\eps,\delta}(c_\eps,u_\eps,z_\eps) = C_0 <\infty,
\end{equation}
allowing us to freely take subsequences and not worry if we have changed the limiting energy---which we typically do as necessary without further comment. For instance, we may suppose $(c_\eps,u_\eps,z_\eps)$ converge pointwise almost everywhere.

We proceed by bounding each term of the energy individually. Note that each function in $(c_\eps,u_\eps,z_\eps)$ is continuous since $H^1(\Omega) \hookrightarrow C^{1/2}(\Omega).$ 

\emph{Step 1 (Lower bound for the crack energy).}
We measure the maximum damage in a region $I\subset \Omega$ with the variable 
\begin{equation}\nonumber
m(I):= \liminf_{\eps\to 0}\left(\inf_{x\in I}z_\eps(x)\right).
\end{equation}
First, we note by the continuity and monotonicity of $\phi_\delta = \phi + o_{\delta \to 0}(1)$ that
\begin{equation}\label{eqn:lowBddPhi}
\liminf_{\eps \to 0}\left(\inf_{x\in I}\phi_{\delta}(z_\eps(x))\right)\ge \phi(m(I)).
\end{equation}

The lower bound on $\phi_\delta(z_\eps)$ allows us to show that if $I\subset \Omega$ is an interval with $m(I)>0$, then there exists some constant $s_{m(I)}>0$ such that 
\begin{equation}\label{eqn:cLinftybdd}
\|c_\eps\|_{L^\infty(I)}\le s_{m(I)}
\end{equation}
for all $\eps$ (up to a subsequence). Precisely, given the pointwise convergence of $c_\eps$, we can find $x_0 \in I$ with $\lim_{\eps \to 0}c_\eps(x_0) = 0$ or $1$. We apply \eqref{eqn:MMtrick} and the fundamental theorem of calculus to find that for any point $x\in I$, we have
\begin{equation}\nonumber
|d_W(c_\eps(x)) - d_W(c_\eps(x_0))| \leq \frac{1}{\inf_{x\in I}\phi_\delta(z_\eps(x))}\int_{I}\phi_{\delta}(z_\eps) \left( \frac{1}{\eps} W(c_\eps) + \eps | {c_\eps'}|^2\right) dx.
\end{equation}
Consequently, we find that
$$\limsup_{\eps\to 0}\left(\sup_{x\in I}|d_W(c_\eps(x))|\right) \leq \frac{C_0}{\phi(m(I))} + d_W(1),$$
which proves \eqref{eqn:cLinftybdd} (up to a subsequence) as $d_W(s)\to \pm \infty$ when $s\to \pm \infty$.

With this, we claim that if $I\subset \Omega$ is an interval with $m(I)>0$, then $I\cap ( J_u \cup \partial^*P) = \emptyset$. To see this, note the $H^1(I)$ bound on $u$ following from
\begin{equation}\label{eqn:gradUbound}
\limsup_{\eps \to 0}\int_I |u_\eps'|^2\, dx \leq \frac{2}{m(I)^2}\limsup_{\eps \to 0}\int_I (z_\eps^2+\delta^2)|u_\eps'-ce_0|^2\, dx + 2(|e_0|s_{m(I)})^2\mathcal{L}^1(I) \leq C.
\end{equation} 
Then $u_\eps - \dashint u_\eps \weakly \tilde u \in H^1(\Omega)$. Necessarily, \eqref{eqn:sepFrame2} requires that $I\cap \partial^*\mathcal{P} = \emptyset$, and subsequently, it becomes clear that also $I\cap J_u = \emptyset,$ as desired

Using the converse of the claim proven in the previous paragraph, we have for any open interval $I\subset \Omega$ with  $I\cap ( J_u \cup \partial^*\mathcal{P}) \neq \emptyset,$ then $m(I) = 0.$ In fact, as the same reasoning holds for any subsequence of $\eps\to 0$, we improve this to $\limsup_{\eps\to 0}\left(\inf_{x\in I}z_\eps(x)\right) = 0. $
Now, we take nested intervals $I\subset\subset I'$ with $I \cap ( J_u \cup \partial^*\mathcal{P})\neq \emptyset$. From the prior reasoning, there is $x_{\eps}\in I$ with $z_\eps(x_{\eps})\to 0$. By the convergence of $z_\eps\to 1$, there are also $x_{\eps}^\pm\in I'\setminus I$ with  $x_{\eps}^- < x_{\eps} < x_{\eps}^+$ such that $z_\eps(x_{\eps}^\pm)\to 1.$ Arguing using \eqref{eqn:MMtrick} and the fundamental theorem of calculus, we have that
\begin{equation}\label{eqn:gammaCharge}
\begin{aligned}
\alpha_{\rm frac} = 2d_V(1) \leq & \liminf_{\eps\to 0} \left(\int_{x_{\eps}^-}^{x_{\eps}} \left(\frac{V(z_\eps)}{\delta} + \delta|{z_\eps'}|^2\right) dx + \int_{x_{\eps}}^{x_{\eps}^+} \left(\frac{V(z_\eps)}{\delta} + \delta|{z_\eps'}|^2\right) dx \right) \\
\leq & \liminf_{\eps\to 0} \int_{I'} \left(\frac{V(z_\eps)}{\delta} + \delta|{z_\eps'}|^2\right) dx.
\end{aligned}
\end{equation}

\emph{Step 2 (Lower bound for the elastic energy).}
To bound the elastic energy from below, the idea is to cut out regions where the damage variable drops below a threshold $\gamma$. 
We show that 
\begin{equation}\label{eqn:1DelasticGlobal}
\int_{\Omega}|u'-ce_0|^2\, dx\leq \liminf_{\eps\to 0} \int_{\Omega}(z_\eps^2 + \delta^2)|u'_\eps-c_\eps e_0|^2\, dx. 
\end{equation}
As a technical aid, we consider an interval $\Omega'\subset\subset \Omega$ and, without loss of generality, suppose $\Omega' = (0,1).$
We now divide up the interval $(0,1)$ into $N\in \N$ chunks of equal length $I_{k,N} := (\frac{k}{N},\frac{k+1}{N}]$ where $k = 0,\ldots,N-1$ (with $I_{N-1,N} := (\frac{N-1}{N},1)$). We assume (up to a subsequence in $\eps$, not relabeled) that 
\begin{equation}\label{eqn:additiveAss}
m(I_{k,N}) = \lim_{\eps\to 0}\left(\inf_{x\in I_{k,N}}z_\eps(x)\right)
\end{equation} for all $k$ and $N.$ Note that if $m(I_{k,N}) \leq \gamma$, reasoning as in \eqref{eqn:gammaCharge}, we have that
$$ d_V(1) - d_V(\gamma) \leq \liminf_{\eps \to 0}\int_{I_{k,N}} \left(\frac{V(z_\eps)}{\delta} + \delta|{z_\eps'}|^2\right) dx.$$
Noting that the above display can be added across intervals due to \eqref{eqn:additiveAss} and denoting the bad indices $k$ in $\{0,\ldots,N-1\}$ for which $m(I_{k,N}) \leq \gamma$ by $\mathbb{B}_{\gamma,N}$, we have that
$$\#(\mathbb{B}_{\gamma,N}) \leq \frac{C_0}{d_V(1) - d_V(\gamma)}. $$
Critically, the above bound on indices is independent of $N$.
As the endpoints of the finitely many bad intervals converge up to a subsequence in $N$ and their width vanishes as $N\to \infty$, we can assume that given any $\eta>0$ we have $$\cup_{k\in \mathbb{B}_{\gamma,N}}I_{k,N} \subset J_\gamma + (-\eta,\eta)$$ for all sufficiently large $N$, where $J_\gamma\subset \overline{\Omega'}\subset \Omega$ (this is the only point where we use $\Omega'\subset\subset \Omega$) and $\mathcal{H}^0(J_\gamma)\leq \frac{C_0}{d_V(1) - d_V(\gamma)}.$ (Here we have abused our jump-set notation, but one should think of $J_\gamma$ as the set where $z$ jumps down to $\gamma$.) It follows that for any $\eta>0$,
\begin{equation}\label{eqn:mBelow}
m(\Omega' \setminus (J_\gamma + (-\eta,\eta))) >\gamma.
\end{equation}

Finally, we tie this information together. First, note that $c_\eps \to c$ in $L^2(\Omega' \setminus (J_\gamma + (-\eta,\eta)))$ because one can take advantage of the bound \eqref{eqn:cLinftybdd}. Consequently, we can use the weak convergence of the gradient $u'_\eps \weakly u'$ in $L^2(\Omega' \setminus (J_\gamma + (-\eta,\eta)))$ (a consequence of the reasoning in \eqref{eqn:gradUbound}) along with \eqref{eqn:mBelow} to estimate
\begin{equation}\nonumber
\gamma^2\int_{\Omega' \setminus (J_\gamma + (-\eta,\eta))} |u' -ce_0|^2\, dx \leq  \liminf_{\eps\to 0}\int_{\Omega' \setminus (J_\gamma + (-\eta,\eta))} (z_\eps^2+\delta^2)|u'_\eps -c_\eps e_0|^2\, dx .
\end{equation} 
Taking $\eta \downarrow 0$ and subsequently $\gamma \uparrow 1$, we have
\begin{equation}\nonumber
\int_{\Omega'}|u'-ce_0|^2\, dx\leq  \liminf_{\eps\to 0} \int_{\Omega}(z_\eps^2 + \delta^2)|u'_\eps-c_\eps e_0|^2\, dx. 
\end{equation}
As the above reasoning (and thereby the above inequality) holds for any finite union of intervals given by $\Omega' \subset\subset \Omega$, we obtain \eqref{eqn:1DelasticGlobal}.

\emph{Step 3 (Lower bound for the {interfacial} energy).}
For an open interval $I'\subset \Omega$ with $I'\cap \partial^*\{c=1\}\neq \emptyset$, we prove that
\begin{equation}\label{eqn:phaseCharge}
\alpha_{\rm surf} \leq \liminf_{\eps \to 0}\left[\int_{I'} \phi_\delta(z_\eps)\left(\frac{W(c_\eps)}{\eps} + \eps|{c_\eps'}|^2\right) dx + \int_{I'} \left(\frac{V(z_\eps)}{\delta} + \delta|{z_\eps'}|^2\right) dx \right].
\end{equation}

We argue analogous to Step 1. Take $I\subset \subset I'$ with {$\mathcal{H}^{0}(I\cap \partial^*\{c=1\})=1$ and $\mathcal{H}^{0}(I'\cap \partial^*\{c=1\})=1$}. Up to a subsequence, we can assume that $m(I) = \lim_{\eps \to 0} \inf_{x\in I}z_\eps(x)$ and then choose $x_\epsilon$ such that $z_\eps(x_\eps) \to m(I)$. From the convergence of $z_\eps$ and $c_\eps$, we can find $x_{\eps}^\pm\in I'\setminus I$ with  $x_{\eps}^- < x_{\eps} < x_{\eps}^+$ such that $z_\eps(x_{\eps}^\pm)\to 1$ and (without loss of generality) $c_\eps(x_{\eps}^\pm)\to 1/2\pm 1/2$.
As in \eqref{eqn:gammaCharge}, we have 
\begin{equation}\label{eqn:cFracCharge}
4 \int_{m(I)}^1 \sqrt{V(s)}\, ds= 2(d_V(1)-d_V(m(I))) \leq \liminf_{\eps\to 0} \int_{x_{\eps}^-}^{x_{\eps}^+} \left(\frac{V(z_\eps)}{\delta} + \delta|{z_\eps'}|^2\right) dx  .
\end{equation}
Likewise, using \eqref{eqn:MMtrick} and \eqref{eqn:lowBddPhi}, we have
\begin{equation}\label{eqn:partPhaseCharge}
\phi(m(I))\alpha_{\rm surf} = \phi(m(I))d_W(1) \leq \liminf_{\eps \to 0}\int_{x_{\eps}^-}^{x_{\eps}^+}\phi_\delta(z_\eps)\left(\frac{W(c_\eps)}{\eps} + \eps|{c_\eps'}|^2\right) dx.
\end{equation}
Summing \eqref{eqn:cFracCharge} and \eqref{eqn:partPhaseCharge} together, and using the relation \eqref{ass:liminf}, we recover \eqref{eqn:phaseCharge}.

\emph{Step 4 (Lower bound for the {total surface} energy).}
We take nested intervals $I_k\subset I_k'$ with $I_k'$ disjoint, so that each interval satisfies $\mathcal{H}^0(I_k \cap ( \partial^*\{c=1\} \cup J_u \cup \partial^*\mathcal{P})) = 1$, and $\partial^*\{c=1\} \cup J_u \cup \partial^*\mathcal{P}\subset  \cup_k I_k$. Applying either \eqref{eqn:gammaCharge} or \eqref{eqn:phaseCharge}, we can sum over all of the intervals $I_k'$ to find
\begin{equation}\label{eqn:lowerBoundSurface}
\begin{aligned}
& \alpha_{\rm surf} \mathcal{H}^0(\partial^*\{c=1\} \setminus (J_u \cup \partial^*\mathcal{P})) + \alpha_{\rm frac}\mathcal{H}^0(J_u \cup \partial^*\mathcal{P})  \\
& \qquad \leq \liminf_{\eps \to 0} \left[\int_{\Omega}\phi_\delta(z_\eps)\left(\frac{W(c_\eps)}{\eps} + \eps|{c_\eps'}|^2\right) dx + \int_\Omega \left(\frac{V(z_\eps)}{\delta} + \delta|{z_\eps'}|^2\right)dx \right]. 
\end{aligned}
\end{equation}

\emph{Step 5 (Conclusion).} Putting together \eqref{eqn:1DelasticGlobal} and \eqref{eqn:lowerBoundSurface} concludes the proof of \eqref{eqn:liminfInequality} when ${d=1}$.
\end{proof}

As it will be needed below, we remark that Step 2 of the proof above actually shows that \eqref{eqn:1DelasticGlobal} can be improved to 
\begin{equation}\label{eqn:1DelasticGlobalModified}
\int_{\Omega}|u'-ce_0 -w|^2\, dx\leq \liminf_{\eps\to 0} \int_{\Omega}(z_\eps^2 + \delta^2)|u'_\eps-c_\eps e_0 -w|^2\, dx. 
\end{equation}
for any $w\in L^2(\Omega).$

\begin{remark}[Slicing]\label{rmk:slicing}
For the proof of the $\liminf$ inequality in arbitrary dimension, we introduce the following notation for slicing. The $(d-1)$-dimensional plane $\Pi_{\xi}$ passes through the origin and is orthogonal to $\xi,$ i.e.,
$$\Pi_\xi := \{y\in \R^d: \langle y,\xi \rangle = 0 \}. $$
Given a direction $\xi\in \mathbb{S}^{d-1}$ and a point $y\in \Pi_\xi$, we can then slice $\Omega$ in parallel lines with 
$$\Omega_{\xi,y}:= \{t\in \R : y+t\xi \in \Omega\}. $$
For $u:\Omega \to \R^d$, on each of these slices we define the functions $u_{\xi ,y}:\Omega_{\xi,y} \to \R$ by
$$u_{\xi,y}(t):= \langle u(y+t\xi) ,\xi \rangle.$$
We overload the notation and for a scalar valued function $c: \Omega \to \R $, we define
$$c_{\xi,y}(t):=  c(y+t\xi),$$
and similarly for a partition $\mathcal{P} = \{P_i\}_{i\in \N}$ of $\Omega$, we denote by $\mathcal{P}_{\xi,y} : = \{(P_i)_{\xi,y}\}_{i\in \N}$ the one dimensional restriction of the partition to $\Omega_{\xi,y}$ with
$$(P_i)_{\xi,y} :=\{t\in \R : y+t\xi \in \Omega \cap P_i\}.  $$

For $u\in GSBD^2(\Omega)$, $c\in BV(\Omega;\{0,1\})$, and a Caccioppoli partition $\mathcal{P} = \{P_i\}_{i\in \N}$, for almost every $\xi \in \mathbb{S}^{d-1}$ and for almost every $y\in \Pi_\xi$, we have that $u_{\xi,y} \in GSBV^2(\Omega_{\xi,y})$, $c_{\xi,y}\in BV(\Omega_{\xi,y};\{0,1\})$, and that $\mathcal{P}_{\xi,y}$ is a Caccioppoli partition of $\Omega_{\xi,y}$ with
\begin{equation}\label{eqn:jumpSliceRel}
\begin{aligned}
&  J_{u_{\xi,y}} = \{t\in \Omega_{\xi,y}: y+t\xi \in J_u\}, \\
 \partial^*\{c_{\xi,y} = 1\} = \{t\in \Omega_{\xi,y}: y+t\xi \in & \partial^*\{c=1\}\}, \quad  \text{ and } \quad \partial^*\mathcal{P}_{\xi,y} = \{t\in \Omega_{\xi,y}: y+t\xi \in \partial^*\mathcal{P}\}.
\end{aligned}
\end{equation}
Formally, the above relations say that $J_{u_{\xi,y}}$ is given by $J_u\cap \Omega_{\xi,y}$.
The symmetrized gradient is connected to the slice's gradient via the relation
\begin{equation}\label{eqn:derivativeSlice}
\tfrac{d}{dt} u_{\xi,y} (t) = \langle e(u)(y+t\xi)\xi,\xi\rangle,
\end{equation}
which highlights why the inner product with $\xi$ was necessary in the definition of $u_{\xi,y}.$
 For further information on slicing, we refer to \cite{AmbrosioFuscoPallara} for $BV$ functions and \cite{dalMasoGSBD} for $GSBD$ functions
\end{remark}

We now turn to the proof of the $\liminf$ inequality in arbitrary dimension. With the $1$D case out of the way, the details follow essentially as done by Iurlano \cite[Theorem 4.3]{iurlano_density} for the approximation of the Griffith energy. However, there is a meaningful difference in verifying the convergence of the sliced sequences due to the skew-affine shifts.

\begin{proof}[Proof of the $\liminf$ inequality of Theorem \ref{thm:main} in any dimension] 
Without loss of generality, we assume 
\begin{equation}\nonumber
\lim_{\eps \to 0}E_{\eps,\delta}(c_\eps,u_\eps,z_\eps) = C_0 <\infty,
\end{equation}
and thus, may take subsequences without disturbing the energy.
We use the slicing notation introduced in Remark \ref{rmk:slicing} without further reference.

\emph{Step 1 (Convergence of the sliced functions).}
We take the collection of infinitesimal rigid motions $\{a_{\eps,i}\}_{i,\eps}$ and the Caccioppoli partition $\mathcal{P}$ as in Theorem \ref{thm:compactness}.
We need to understand how the convergence in measure is carried by slices. Of course, the convergence of $(c_\eps)_{\xi,y}$ and $(z_\eps)_{\xi,y}$ to $c_{\xi,y}$ and $z_{\xi,y}$ follows from a direct application of Fubini's theorem. For the displacements $u_\eps$ satisfying \eqref{eqn:displacementConvergence},
we verify that the convergence \eqref{eqn:measConverge1D} holds for the almost every instantiation of $(u_\eps)_{\xi,y}$  with the piecewise constant shift given by \begin{equation}\label{eqn:pwConstShiftSlice}
\sum_{i} (a_{\eps,i})_{\xi,y}\chi_{(P_i)_{\xi,y}}. 
\end{equation} Further, we show the piecewise constant shift associated with the partition $\mathcal{P}_{\xi,y}$ satisfies \eqref{eqn:sepFrame2}.

First, for all $t\in \Omega_{\xi,y}$, we have that
\begin{equation}\label{eqn:pwConsts}
 (a_{\eps,i})_{\xi,y}(t) = \langle A_{\eps,i}(y+t\xi) + b_{\eps,i}, \xi \rangle = \langle A_{\eps,i}(y) + b_{\eps,i}, \xi \rangle,
\end{equation}  as $A\xi \in \Pi_\xi$ for any skew-affine matrix $A$, so that \eqref{eqn:pwConstShiftSlice} is infact a piecewise constant function on $\Omega_{\xi,y}$.
The convergence \eqref{eqn:measConverge1D} holds by Fubini's as before since
\begin{equation}\label{eqn:sliceU}
\Big(u_\eps - \sum_{i} a_{\eps,i}\chi_{P_i}\Big)_{\xi,y} = (u_\eps)_{\xi,y} -  \sum_{i} (a_{\eps,i})_{\xi,y}\chi_{(P_i)_{\xi,y}}.
\end{equation}

The rest of this step is dedicated to showing that \eqref{eqn:sepFrame2} is satisfied for the grotesque piecewise constant function in \eqref{eqn:pwConstShiftSlice} for almost every $\xi\in \mathbb{S}^{d-1}$ and almost every $y\in \Pi_\xi$. 
 From \eqref{eqn:separatingFrames}, for a pair $(i,j)$ with $i\neq j$, we see that if 
\begin{equation}\label{eqn:bBounded}
\limsup_{\eps\to 0} |b_{\eps,i} - b_{\eps,j}|<\infty,
\end{equation} 
then necessarily ${\lim_{\eps\to 0}} |A_{\eps,i} - A_{\eps,j}| = \infty.$ From this last condition, {by taking a subsequence in $\eps$,} one can show that for almost every $\xi\in \mathbb{S}^{d-1}$ it holds that
\begin{equation}\label{eqn:xiFurther}
{\lim_{\eps\to 0} } |(A_{\eps,i} - A_{\eps,j})(\xi)| = \infty.
\end{equation} 
Precisely, denote the $l$-th row of $A_{\eps,i} - A_{\eps,j}$ by $(A_{\eps,i} - A_{\eps,j})_l$; then since ${\lim_{\eps\to 0}} |A_{\eps,i} - A_{\eps,j}| = \infty$, there is some row $l_0$ such that $|(A_{\eps,i} - A_{\eps,j})_{l_0}|\to \infty$ and $\frac{(A_{\eps,i} - A_{\eps,j})_{l_0}}{|(A_{\eps,i} - A_{\eps,j})_{l_0}|}\to a_{(i,j)}\in \mathbb{S}^{d-1}$ up to a subsequence; with this, \eqref{eqn:xiFurther} holds for $\xi \not \in \Pi_{a_{(i,j)}}$.
 Consequently, we restrict $\xi$ so that if \eqref{eqn:bBounded} holds, then \eqref{eqn:xiFurther} also holds. 
Similarly, for any pair $(i,j)$ such that 
\begin{equation}\label{eqn:ijlimits}
\limsup_{\eps\to 0} |b_{\eps,i} - b_{\eps,j}| = \infty,
\end{equation} 
up to taking a subsequence in $\eps$, we have for almost every $\xi\in \mathbb{S}^{d-1}$ that
\begin{equation}\label{eqn:xiRestrict}
{\lim_{\eps\to 0} } |\langle b_{\eps,i} - b_{\eps,j},\xi\rangle| = \infty.
\end{equation} 
Precisely, one may take subsequences so that $\lim_{\eps\to 0} |b_{\eps,i} - b_{\eps,j}| = \infty$ and
we define $\frac{b_{\eps,i} - b_{\eps,j}}{|b_{\eps,i} - b_{\eps,j}|} \to {b_{(i,j)}}\in \mathbb{S}^{d-1};$ with this, \eqref{eqn:xiRestrict} holds for $\xi \not \in \Pi_{b_{(i,j)}}$.
In summary, we may take a subsequence in $\eps$ so that, for almost every $\xi \in \mathbb{S}^{d-1}$, if \eqref{eqn:bBounded} holds then so does \eqref{eqn:xiFurther} and, likewise, if \eqref{eqn:ijlimits} holds then so does \eqref{eqn:xiRestrict}.

We \textbf{claim} that \eqref{eqn:sepFrame2} with constants as in \eqref{eqn:pwConsts} holds for each pair $(i,j)$ on $\Pi_\xi$ for any $\xi$ as in the preceding paragraph, up to a subsequence in $\eps$ (depending on $\xi$). 
{Supposing \eqref{eqn:sepFrame2} does not hold for some pair $(i,j)$,} we inductively construct subsequences: one of these subsequences will satisfy \eqref{eqn:sepFrame2} for $(i,j)$, or we will manage to find a collection of vectors $\{y_k\}_{k=0}^{d-1}\subset \Pi_\xi$ such that $\{y_k -y_0\}_{k=1}^{d-1}$ forms a linearly independent basis of $\Pi_\xi$ and
\begin{equation}\label{eqn:falseCondition}
\limsup_{\eps\to 0} |(a_{\eps,i})_{\xi,y_k} -(a_{\eps,j})_{\xi,y_k}|<\infty \quad \text{ for each } k = 0,\ldots, d-1,
\end{equation}
which will later be used to find a contradiction.
{To start with, since \eqref{eqn:sepFrame2}} does not hold on a set $Y \subset \Pi_\xi$ with $\mathcal{H}^{d-1}(Y)>0$ for the pair $(i,j)$, we may find $y_0\in Y$ satisfying the inequality in \eqref{eqn:falseCondition} by taking a subsequence in $\eps$. For this new subsequence, either \eqref{eqn:sepFrame2} holds for the pair $(i,j)$ $\mathcal{H}^{d-1}$-almost everywhere on $\Pi_\xi$ or there is a set $Y_0 \subset \Pi_\xi$ with $\mathcal{H}^{d-1}(Y_0)>0$ on which it fails. In the former case, we are done and may proceed to the next step. In the latter case, we take a further subsequence in $\eps$ to find $y_1\in Y_0$ satisfying the inequality in \eqref{eqn:falseCondition} with $\{y_1-y_0\}$ a linearly independent basis. Continuing in this manner, one either finds that \eqref{eqn:sepFrame2} holds for $(i,j)$ up to taking a subsequence in $\eps$ or constructs the desired vectors $\{y_k\}_{k=0}^{d-1}$  satisfying \eqref{eqn:falseCondition}.

 We now use the constructed vectors to derive a contradiction. Recalling \eqref{eqn:pwConsts}, we may subtract the condition \eqref{eqn:falseCondition} for different $y_k$ and $y_{0}$ to find
$$\limsup_{\eps \to 0}|\langle (A_{\eps,i} - A_{\eps,j})(y_k - y_{0}),\xi \rangle| \leq C <\infty. $$
Given that $\{y_k -y_0\}_{k=1}^{d-1}$ forms a basis of $\Pi_\xi$ and that $A_{\eps,i}$ and $A_{\eps,j}$ are skew-affine, we have
$$\limsup_{\eps \to 0}|(A_{\eps,i} - A_{\eps,j})(\xi) | \leq C <\infty ,$$
and so as to avoid a contradiction with \eqref{eqn:xiFurther}, it must also hold that
$$ \lim_{\eps\to 0} |b_{\eps,i} - b_{\eps,j}| = \infty.$$
Using the above two displays and recalling that $\xi$ satisfies \eqref{eqn:xiRestrict}, we have
\begin{equation}
\liminf_{\eps \to 0}| (a_{\eps,i})_{\xi,y_k} -  (a_{\eps,j})_{\xi,y_k}|\geq \liminf_{\eps \to 0}|\langle b_{\eps,i} - b_{\eps,j},\xi\rangle| - \limsup_{\eps \to 0}|\langle y_k, (A_{\eps,i} - A_{\eps, j})(\xi) \rangle| = \infty,
\end{equation}
contradicting \eqref{eqn:falseCondition} and concluding the claim.

\emph{Step 2 (Lower bound for the elastic energy).}
{Take a countable dense collection $\Xi$ of $\mathbb{S}^{d-1}$ and a single subsequence of $(c_\eps,u_\eps,z_\eps)$ so that for every $\xi\in \Xi$, the convergences of Step $1$ and \eqref{eqn:jumpSliceRel} hold for almost every $y \in \Pi_\xi$; this is possible via diagonalization.
For $\xi \in \Xi$,} we prove that 
\begin{equation}\label{eqn:elasticLSC}
\int_{\Omega} |\langle (e(u) - ce_0)\xi,\xi\rangle- w|^2\, dx \leq  \liminf_{\eps\to 0} \int_{\Omega}(z_\eps^2 + \delta^2)|\langle (e(u_\eps) - c_\eps e_0)\xi,\xi\rangle -w|^2\, dx
\end{equation}
for all $w\in L^2(\Omega)$. But this follows immediately noting that by \eqref{eqn:derivativeSlice} and Fubini's theorem, we have
\begin{align*}\nonumber
& \int_{\Omega} (z_\eps^2 + \delta^2)|\langle (e(u_\eps) - c_\eps e_0)\xi,\xi\rangle -w|^2\, dx  \\
 &\qquad  = \int_{\Pi_\xi}\left(\int_{\Omega_{\xi,y}}((z_\eps)_{\xi,y}^2 + \delta^2)|{\tfrac{d}{dt}} (u_\eps)_{\xi,y} - (c_\eps)_{\xi,y} \langle e_0\xi,\xi\rangle -w_{\xi,y}|^2\, dt \right) d\mathcal{H}^{d-1}(y).
\end{align*}
  Applying \eqref{eqn:1DelasticGlobalModified} on the slices $\Omega_{\xi,y}$ and then using Fatou's lemma gives \eqref{eqn:elasticLSC}.

  Using \eqref{eqn:elasticLSC}, we can show that
  \begin{equation}\label{eqn:weakConvGrads}
  \sqrt{z_\eps^2 + \delta^2}(e(u_\eps) - c_\eps e_0) \weakly (e(u) - c e_0) \text{ in }L^2(\Omega; \R^{d\times d}_{\rm sym}).   
  \end{equation}
  Precisely, supposing that $\sqrt{z_\eps^2 + \delta^2}(e(u_\eps) - c_\eps e_0) \weakly M$ in $L^2(\Omega;\R^{d\times d}_{\rm sym})$ and $z_\eps\to 1$ in $L^2(\Omega)$ (which we may suppose up to a subsequence), we can rewrite \eqref{eqn:elasticLSC} as 
  \begin{align*}\nonumber
  &\int_{\Omega} |\langle (e(u) - ce_0)\xi,\xi\rangle|^2\, dx  - \int_{\Omega} 2 w \langle (e(u) - ce_0 - M)\xi,\xi\rangle \, dx  \\
    & \qquad\qquad \leq \liminf_{\eps\to 0} \int_{\Omega}(z_\eps^2 + \delta^2)|\langle (e(u_\eps) - c_\eps e_0)\xi,\xi\rangle |^2\, dx \leq C_0.
  \end{align*}
  As $w\in L^2(\Omega)$ is arbitrary, this is only possible if $\langle (e(u) - ce_0 - M)\xi,\xi\rangle = 0$ for almost every $x\in \Omega$ and all $\xi \in \mathbb{S}^{d-1}$ (here we have used the density of {$\Xi$}). Diagonalization of the symmetric matrices shows that $e(u) - ce_0 - M = 0$, thereby recovering \eqref{eqn:weakConvGrads}.
  The weak convergence then implies lower semi-continuity of the elastic energy:
  \begin{equation}\label{eqn:liminfElasticFull}
  \int_{\Omega}\CC{(e(u) - ce_0)} \, dx \leq \liminf_{\eps \to 0} (z_\delta^2 + \delta^2)\int_{\Omega}\CC{(e(u_\eps) - c_\eps e_0)} \, dx.
  \end{equation}
  
  \emph{Step 3 (Lower bound for the {total surface} energy).}
  {Take $\xi\in \Xi$ as in Step 2.} As for \eqref{eqn:elasticLSC}, we may apply the $1$D $\liminf$ bound and in particular \eqref{eqn:lowerBoundSurface} to find
  \begin{equation}\nonumber
  \begin{aligned}
&\int_{\Pi_\xi}\left( \int_{\Omega_{\xi,y}} \alpha_{\rm surf} \mathcal{H}^0(\partial^*\{c_{\xi,y}=1\} \setminus (J_{u_{\xi,y}} \cup \partial^*\mathcal{P}_{\xi,y})) + \alpha_{\rm frac}\mathcal{H}^0(J_{u_{\xi,y}} \cup \partial^*\mathcal{P}_{\xi,y}) \, dt \right) d\mathcal{H}^{d-1}(y) \\
& \qquad  \leq  \liminf_{\eps\to 0}\left[\int_{\Omega}\phi_\delta(z_\eps)\left(\frac{W(c_\eps)}{\eps} + \eps|\nabla c_\eps|^2\right) dx + \int_\Omega \left(\frac{V(z_\eps)}{\delta} + \delta|\nabla z_\eps|^2\right)dx \right]
\end{aligned}
\end{equation}
Using \eqref{eqn:jumpSliceRel} and the coarea formula \cite{AmbrosioFuscoPallara}, we can rewrite this as
  \begin{equation}\label{eqn:liminfSurf}
  \begin{aligned}
&\alpha_{\rm surf }\int_{\Omega\cap \big(\partial^*\{c=1\} \setminus (J_{u} \cup \partial^*\mathcal{P})\big)}  |\langle \nu(x), \xi \rangle|\, d\mathcal{H}^{d-1} + \alpha_{\rm frac }\int_{\Omega \cap \big(J_{u} \cup \partial^*\mathcal{P}\big)}  |\langle \nu(x), \xi \rangle|\, d\mathcal{H}^{d-1} \\
& \qquad  \leq  \liminf_{\eps\to 0}\left[\int_{\Omega}\phi_\delta(z_\eps)\left(\frac{W(c_\eps)}{\eps} + \eps|\nabla c_\eps|^2\right) dx + \int_\Omega \left(\frac{V(z_\eps)}{\delta} + \delta|\nabla z_\eps|^2\right)dx \right],
\end{aligned}
\end{equation}
  where $\nu$ is the underlying measure-theoretic normal associated to the surface being integrated.
 As \eqref{eqn:liminfSurf} holds with $\Omega$ replaced by any open subset $\Omega'$, a classical localization argument for measures (see, e.g., \cite[Theorem 1.16 and Section 4.1]{braidesFreeDiscont}) effectively choosing $\xi = \nu(x)$ at each point shows that
  \begin{equation}\label{eqn:liminfSurfFull}
  \begin{aligned}
&\alpha_{\rm surf }\mathcal{H}^{d-1}\big(\partial^*\{c=1\} \setminus (J_{u} \cup \partial^*\mathcal{P})\big)   + \alpha_{\rm frac }\mathcal{H}^{d-1}\big(J_{u} \cup \partial^*\mathcal{P}\big)  \\
& \qquad  \leq  \liminf_{\eps\to 0}\left[\int_{\Omega}\phi_\delta(z_\eps)\left(\frac{W(c_\eps)}{\eps} + \eps|\nabla c_\eps|^2\right) dx + \int_\Omega \left(\frac{V(z_\eps)}{\delta} + \delta|\nabla z_\eps|^2\right)dx \right].
\end{aligned}
\end{equation}

  \emph{Step 4 (Conclusion).} The inequalities \eqref{eqn:liminfElasticFull} and \eqref{eqn:liminfSurfFull} conclude the proof of the $\liminf$ inequality.
\end{proof}

\section{The $\limsup$ inequality}\label{sec:limsup}

In this section, we show that the lower bound found via the $\liminf$ inequality is optimal, and for a given target function $(c,u)\in BV(\Omega;\{0,1\})\times GSBD^2(\Omega)$, we construct a recovery sequence $(c_\eps,u_\eps, z_\eps)$ obtaining the lower bound as its energetic limit. 
Before this, we prove two auxiliary results: the first is a density result for sets of finite perimeter, and the second is a technical lemma showing that the {interfacial} energy contained in the crack and phase boundary's overlap may be excluded in the limit.

As is often the case in constructing recovery sequences, our proof uses density results to approximate $c$ and $u$ by more regular functions, which can be explicitly used to define the recovery sequence. For both functions, we require relatively strong approximations for the discontinuity sets to ensure that the phase boundary energy $\mathcal{H}^{d-1}(\partial^* \{c=1\}\setminus J_u)$ is stable with respect to the approximation.
For $GSBD^2$ functions, existing results suffice. For sets of finite perimeter, we simplify an approximation argument used by Bungert and the first author \cite{bungertStinson2024} that relies on an $SBV$ approximation result of De Philippis et al. \cite{dePhilippis2017}.

 The second technical lemma is particular to our case. The point is that even when the boundary $\partial^* \{c=1\}$ and the jump-set $J_u$ are smooth sets, the intersection between them may not be a particularly regular set. To recover the energy $\mathcal{H}^{d-1}(\partial^*\{c=1\}\setminus J_u)$ with our smooth approximation, we show that
 $$\mathcal{H}^{d-1}(\{\operatorname{dist}(x,\{c=1\}) = r\} \setminus \{\operatorname{dist}(x,J_u)<3r\}) \xrightarrow[r \to 0]{} \mathcal{H}^{d-1}(\partial^*\{c=1\}\setminus J_u). $$
  For $\Gamma$-convergence of the Cahn--Hilliard energies, i.e., without an elastic interaction, a simpler statement is used for $\partial^* \{c=1\}$ given by a $C^2$ surface.  

\begin{lemma}[Approximation for a set of finite perimeter]\label{lem:setApprox}
Given $c \in BV(\Omega ; \{0,1\})$ and $\eta>0$, we can find a set of finite perimeter $A$ such that 
 \begin{equation}\label{eqn:BVnormControl}
 \|c - \chi_A\|_{BV(\Omega)} \leq \eta
 \end{equation}
 and there is a closed set $N\subset\subset \Omega$ with $\mathcal{H}^{d-1}(N) = 0$ such that for every $x\in (\partial A\cap \Omega) \setminus N$, there is a radius $r_x>0$ such that $\partial^* A \cap B(x,r_x) = \partial A\cap B(x,r_x)$ is a $C^1$ {manifold}. Further $A\subset\Omega$ can be extended to a set still denoted by $A\subset{\R^d}$ such that $\partial A\subset {\R^d}$ is a $C^1$ {manifold} in a neighborhood of $\partial \Omega$, and $\partial A$ has transverse intersection with $\partial \Omega$ in the sense that
$\mathcal{H}^{d-1}(\partial A\cap \partial \Omega)  = 0 .$
\end{lemma}

\begin{proof}
The idea of the proof is to approximate $c$ by an $SBV$ function $v$ that has a regular jump-set, and then define $A$ as the super level-set $\{v>s\}$ for a well-chosen $s$. The approximation $v$ is constructed by applying a density result, for which $v$ has a continuous representative on either side of its jump-set. Applying Sard's theorem, we will see that the boundary of the super level-set is regular, as needed. The estimate on the $BV$ norm of $c- \chi_{\{v>s\}}$ follows from direct estimates. 

\emph{Step 1 ($SBV$ approximation).} Fix any $p\in (d,\infty).$
Applying \cite[Theorem C]{dePhilippis2017} to $c-1/2$, we find a $C^1$ manifold $M\subset\subset \Omega$ {(possibly with boundary)} and a function $v\in C^1(\Omega\setminus M)\cap SBV^p(\Omega)$ with
\begin{equation}\label{eqn:DePhilippisApprox}
\begin{aligned}
\|c - v\|_{BV(\Omega)} & \leq \eta \\
\| v-1/2\|_{L^\infty(\Omega)} & \leq 1/2 \\
\mathcal{H}^{d-1}(\partial^*\{c=1\} \triangle J_{v})&\leq \eta \\
\mathcal{H}^{d-1}(M\setminus J_{v}) & = 0 \text{ with }J_{v}\subset M
\end{aligned}
\end{equation}
Note the $L^\infty$ bound holds due to the comment at the top of \cite[pg. 372]{dePhilippis2017}, and we will use that $0\leq v\leq 1$ without further reference. Anytime we write $\partial M$, this refers to the manifold boundary of $M.$

We briefly note that on either side of $M$, or on $\partial \Omega$, the function $v$ coincides with a locally defined continuous representative. Precisely, we cover $M\setminus \partial M$ by countably many balls $B_i \subset \subset \Omega$, $i\in \N$, centered on $M$ with $B_i\cap M $ given by the graph of a $C^1$ function. 
Each $B_i\setminus M$ can be written as the union of two (almost) hemispheres $H^{\pm}_i$ (super- and subgraphs). As $J_v \subset M$ and $v \in SBV^p(\Omega)$, we have $v\llcorner H^\pm_i \in W^{1,p}(H^\pm_i) \hookrightarrow C^{1-d/p}(\overline{H^\pm_i})$ by the Morrey embedding theorem. We use $v_i^\pm \in C^0(\overline{H^\pm_i})$ to denote the continuous extension of $v\llcorner H^\pm_i$ defined also on $\partial H^\pm_i$. Though we do not carry out the details, a local reflection and mollification argument on $\partial \Omega$ (around which $v$ is Sobolev) shows that we may suppose $v$ has a $C^\infty$ extension in a neighborhood of the boundary $\partial \Omega$.

\emph{Step 2 (Regularity of level-sets).}
We will choose $A = \{v>s\}$ for a well chosen $s$, as determined in Step 3 below. Here we show that for almost every choice of $s\in (0,1)$ the boundary $\partial \{v>s\}$ is regular as in the lemma statement. The set $N$ effectively comes from $\partial M$ and the intersection of $\overline{\{v=s\}\setminus M}$ with $M$. As $\partial M$ is closed and $\mathcal{H}^{d-1}$-negligible, we only have to address the size of the intersection of the level-set with $M\setminus \partial M.$
The regularity across $\partial \Omega$ will follow as in the interior of $\Omega\setminus  M$ since $v$ was defined so that it can be smoothly extended outside of $\Omega$, and as such, we do not spell out the details (note the transverse intersection condition follows from the same reasoning as \eqref{eqn:hemi1} below).

\emph{Substep 2.1 (Interior regularity).}
Since $v\in C^1(\Omega\setminus M)$, we may apply Sard's theorem to find that for almost every $s \in (0,1)$, the set $\partial \{v>s\}\cap (\Omega\setminus M) = \{v=s\} \cap (\Omega\setminus M)$ is (locally) a $C^1$ surface in $\Omega\setminus M$.

\emph{Substep 2.2 (Intersection of the level-set with $M$).}
Fix one of the hemispheres $H_i^+$. As $v$ is continuous up to $M$ from inside $H_i^+$, $$\overline{\{v = s\}\cap H_i^+}\cap M \cap B_i \subset \{v^+_i = s\} \cap M \cap B_i.$$  
The important takeaway from this relation is that we may cover these surface intersections by pairwise disjoint sets; precisely, 
$$\left(\{v^+_i = s\}\cap M \cap B_i\right) \cap \left(\{v^+_i = t\}\cap M \cap B_i\right) =\emptyset$$
for all $s,t\in (0,1)$ with $s\neq t$. Since $\mathcal{H}^{d-1}(M)<\infty$, only a countable number of these sets may be `charged' so that 
\begin{equation}\label{eqn:hemi1}
 \mathcal{H}^{d-1}(\{v^+_i = s\}\cap M \cap B_i) = 0
 \end{equation}
for almost every $s\in (0,1)$. 

Uniting this insight for every hemisphere $H^{\pm}_i$,
we define $$N = N(s) : = \partial M \cup \bigg(\bigcup_{i,\pm}(\{v^\pm_i = s\}\cap M \cap B_i)\bigg).$$ As the hemispheres $H^\pm_i$ are countable and $\mathcal{H}^{d-1}(\partial M) = 0$, we improve \eqref{eqn:hemi1} to
\begin{equation}\nonumber
\mathcal{H}^{d-1}(N) = 0
\end{equation}
for almost every $s\in (0,1)$. Note $N$ is closed as the $B_i$ cover $M\setminus \partial M.$

\emph{Substep 2.3 (Regularity away from $N$).} By Substep 2.1, it remains to show that if $x \in \partial\{v>s\}\cap M \setminus N$, then there is some $r_x >0$ such that $\partial\{v>s\}\cap B(x,r_x)$ is a $C^1$ surface. In fact, we show there is $r_x>0$ with 
\begin{equation}\label{eqn:surfaceIdent}
\partial\{v>s\} \cap B(x,r_x) = M \cap B(x,r_x) .
\end{equation} {Take $i$ such that $x\in B_i$ also.} Since $x\not \in N$, we have that $v^\pm_i(x)\neq s$. Necessarily the values must be on opposite sides of $s$, so that without loss of generality, $v^+_i(x)>s>v^-_i(x)$. Then \eqref{eqn:surfaceIdent} follows immediately from the continuity of $v^\pm_i$.

\emph{Step 3 (Estimate on the $BV$ norm).}
We define $A_s : = \{v>s\}$. We restrict our consideration to $s \in (3/8,5/8)$ and show that there is at least one such $s$ such that $A = A_s$ satisfies estimate \eqref{eqn:BVnormControl} and the regularity of Step 2. So as to make notation less cumbersome, we use that $J_c = \partial^*\{c=1\}$.

The estimate for the $L^1$ norm follows from \eqref{eqn:DePhilippisApprox} with
\begin{equation}
\int_{\Omega}|c - \chi_{A_s}|\, dx \leq   \int_{\{c=1\}\setminus A_s} \frac{|c-v|}{1-s}\, dx + \int_{ A_s\setminus \{c=1\}} \frac{|c-v|}{s}\, dx\leq \frac{8}{3}\int_{\Omega}|c-v|\, dx \leq C\eta.
\end{equation}

The $BV$ seminorm is more involved. First, note 
\begin{equation}\label{eqn:TVsplit}
|D(c- \chi_{A_s})|(\Omega) = \mathcal{H}^{d-1}(J_c \triangle \partial^* A_s) + 2\mathcal{H}^{d-1}(J_c \cap \partial^* A_s\cap \{n_c \neq n_{A_s}\}), 
\end{equation}
where $n_c$ and $n_{A_s}$ are the measure theoretic inner normals for $\{c=1\}$ and $A_s$, respectively. To estimate the Hausdorff measure of the sets on the right side of \eqref{eqn:TVsplit}, we case on whether we are inside or outside of $J_v$. 

For the first term on the right-hand side of \eqref{eqn:TVsplit}, we have
\begin{equation}\label{eqn:TVsplit2}
\begin{aligned}
\mathcal{H}^{d-1}(J_c \triangle \partial^* A_s) &= \mathcal{H}^{d-1}(J_c \setminus (\partial^* A_s \cup J_v)) + \mathcal{H}^{d-1}((\partial^* A_s \cap J_v) \setminus J_c) \\
& \quad + \mathcal{H}^{d-1}((J_c \cap J_v )\setminus \partial^* A_s) + \mathcal{H}^{d-1}(\partial^* A_s \setminus (J_c \cup J_v))
\end{aligned}
\end{equation}
By the third inequality of \eqref{eqn:DePhilippisApprox}, the first and second term on the right-hand side are bounded from above by $\eta$. To control the third term on the right-hand side of \eqref{eqn:TVsplit2}, note that if $x \in J_v\setminus \partial^* A_s$ and we let $v^+(x)>v^-(x)$ denote the (ordered) one-sided trace values on $J_v$, then either $1\geq v^+(x)>v^-(x)\geq s$ or $s\geq v^+(x)>v^-(x)\geq 0$. Given that $s\in (3/8,5/8)$, it follows that $|1- (v^+(x) - v^-(x))|>1/4$. From this, we have
\begin{equation}\nonumber
\begin{aligned}
\mathcal{H}^{d-1}((J_c \cap J_v )\setminus \partial^* A_s) &\leq \mathcal{H}^{d-1}(J_c \cap J_v \cap \{|1- (v^+ - v^-)|>1/4\}) \\
& \leq 4\int_{J_c \cap J_v}|1- (v^+ - v^-)| \, d\mathcal{H}^{d-1} \\
& \leq 4 \int_{J_c \cap J_v}|(c - v)^+ - (c-v)^-| \, d\mathcal{H}^{d-1} \leq 4\|c-v\|_{BV(\Omega)}\leq 4\eta,
\end{aligned}
\end{equation}
where we use that $c^+ = 1$ and $c^- = 0$ $\mathcal{H}^{d-1}$-almost everywhere on $J_c \cap J_v$  and $\pm (c-v)^\pm \geq \pm (c^\pm - v^\pm)$. The fourth term on the right-hand side of \eqref{eqn:TVsplit2} may be estimated using the coarea formula as follows:
\begin{equation}\label{eqn:coareaEst}
\int_{\R}\mathcal{H}^{d-1}(\partial^*A_s\setminus (J_c\cup J_v))\, ds = \|\nabla v\|_{L^1(\Omega)}\leq \|c-v\|_{BV(\Omega)}\leq \eta,
\end{equation}
where we have use that the absolutely continuous part of $Dc $ is $0$. Thus, \eqref{eqn:coareaEst} implies there is a $\tfrac12$-fraction of $s\in (3/5,5/8)$ such that
$$ \mathcal{H}^{d-1}(\partial^*A_s\setminus (J_c\cup J_v)) \leq 8\eta.$$ With this last estimate, we see that for a $\tfrac12$-fraction of $s\in (3/8,5/8)$ the quantity in \eqref{eqn:TVsplit2} is controlled by $C\eta$ for some $C >0$.

For the second term on the right-hand side of \eqref{eqn:TVsplit}, we see as before that it suffices to estimate $\mathcal{H}^{d-1}(J_c \cap \partial^* A_s\cap J_v \cap \{n_c \neq n_{A_s}\})$ by $\eta$. For this, we note that $\mathcal{H}^{d-1}$-almost every $x\in J_c \cap J_v \cap \{n_c \neq n_{A_s}\}$ one has $(c-v)^+(x) - (c-v)^-(x) = 1 -v^-(x) - (0 - v^+(x))\geq 1$. Consequently, it follows that
\begin{equation}\nonumber
\begin{aligned}
\mathcal{H}^{d-1}(J_c \cap \partial^* A_s\cap J_v \cap \{n_c \neq n_{A_s}\}) \leq \int_{J_c \cap J_v} |(c-v)^+ - (c-v)^-|\, d \mathcal{H}^{d-1}\leq \|c-v\|_{BV(\Omega)}\leq \eta.
\end{aligned}
\end{equation} 
Synthesizing the above estimates with \eqref{eqn:TVsplit}, we are able to find $s\in (3/8,5/8)$ such that $A_s$ satisfies the regularity of Step 2 and \eqref{eqn:BVnormControl} once the constant is absorbed into $\eta$.
\end{proof}

\begin{lemma}[Exclusion of the jump-energy]\label{lem:jumpExclusion}
Suppose $A\subset\Omega$ can be extended to a set still denoted by $A\subset{\R^d}$ such that $\partial A\subset {\R^d}$ is a $C^1$ manifold in a neighborhood of $\Omega$, and further $\partial A$ has transverse intersection with $\partial \Omega$ in the sense that
$\mathcal{H}^{d-1}(\partial A\cap \partial \Omega)  = 0 .$
 Then for any $M\subset \Omega$, it holds that
\begin{equation}\label{eqn:jumpExclusion1}
\limsup_{r\to 0}\mathcal{H}^{d-1}\big((\{\operatorname{dist}(x,A) = r\}\cap \Omega) \setminus \{\operatorname{dist}(x,M)<3r\}\big) \leq \mathcal{H}^{d-1}((\partial A \cap \Omega) \setminus M),
\end{equation}
where $\operatorname{dist}(x,A)$ refers to the extended set $A\subset \R^d$.
\end{lemma}
\begin{proof}
\emph{Step 1 (Reduction to compactly contained sets).}
We will show that for any $\Omega ' \subset\subset \Omega$
\begin{equation}\label{eqn:jumpExclusion}
\limsup_{r\to 0}\mathcal{H}^{d-1}\big((\{\operatorname{dist}(x,A) = r\}\cap \Omega') \setminus \{\operatorname{dist}(x,M)<3r\}\big) \leq \mathcal{H}^{d-1}((\partial A \cap \Omega) \setminus M).
\end{equation}
In fact, since $\partial A$ is regular in a neighborhood of $\partial \Omega$, the same argument will show that for small $\eta>0$
\begin{equation}\nonumber
\limsup_{r\to 0}\mathcal{H}^{d-1}\big(\{\operatorname{dist}(x,A) = r\}\cap \{\operatorname{dist}(x,\partial \Omega)<\eta\} \big) \leq \mathcal{H}^{d-1}(\partial A \cap \{\operatorname{dist}(x,\partial \Omega)<2\eta\}).
\end{equation}
Assuming that the above equations hold, we take $\Omega' = \{\operatorname{dist}(x, \Omega^c)>\tfrac12\eta\}$, apply the above two inequalities, and then send $\eta\to 0$ to find
\begin{equation}\nonumber
\begin{aligned}
\limsup_{r\to 0}\mathcal{H}^{d-1}& \big((\{\operatorname{dist}(x,A) = r\}\cap \Omega) \setminus \{\operatorname{dist}(x,M)<3r\}\big) \\
& \leq \mathcal{H}^{d-1}((\partial A \cap \Omega) \setminus M) + \limsup_{\eta\to 0}\mathcal{H}^{d-1}(\partial A \cap \{\operatorname{dist}(x,\partial \Omega)<2\eta\}) \\
& \leq \mathcal{H}^{d-1}((\partial A \cap \Omega) \setminus M) + \mathcal{H}^{d-1}(\partial A \cap \partial \Omega),
\end{aligned}
\end{equation}
which gives \eqref{eqn:jumpExclusion1} due to the transverse intersection assumption. We now prove \eqref{eqn:jumpExclusion}.

\emph{Step 2 (Reduction to cubes).} We \textbf{claim} that \eqref{eqn:jumpExclusion} holds in general if it holds when $\Omega'$ and $\Omega$ are both replaced by $Q$, where $Q$ is any cube such that in $(1+\theta)Q$, for some $\theta\in (0, 1)$, $\partial A$ is the graph of a $C^1$ function (up to rotation and passing through the center of the cube) with Lipschitz constant less than $\frac12$. With this claim in hand, we prove the general version of \eqref{eqn:jumpExclusion}.

We apply the Morse measure covering theorem \cite{fonseca2007modern}[Theorem 1.147] to find a disjoint (almost-)cover of $\partial A\cap \overline{\Omega'}$ by countably many cubes $\{Q_i\}_{i\in \N}$, with side length $s_i>0$, such that $\partial A\cap 2Q_i$ is given by the graph (up to a rotation and passing through the center of the cube) of a $C^1$ function having Lipschitz constant less than $\frac12$ and 
$$\mathcal{H}^{d-1}(\partial A\cap \overline{\Omega'} \setminus \cup_{i\in \N}Q_i ) = 0 .$$ As each of the graphs is Lipschitz, we have for any $0 < \theta \ll 1$ that
$$\mathcal{H}^{d-1}(\partial A\cap ((1+\theta)Q_i\setminus Q_i))\leq 4(d-1)\theta s_i^{d-1} \quad \text{ and }\quad   s_i^{d-1} \leq \mathcal{H}^{d-1}(\partial A \cap Q_i),$$ where we have used the area formula and that $(1+\theta)^{d-1}\leq 1+ 2(d-1)\theta$ for sufficiently small $\theta>0.$
We take $I \subset \N$ to be a finite collection of indices such that $\partial A\cap \overline{\Omega'} \subset\cup_{i\in I}(1+\theta)Q_i$. Then we apply \eqref{eqn:jumpExclusion} with both $\Omega'$ and $\Omega$ replaced by $(1+\theta)Q_i$ to find that
\begin{equation}\nonumber
\begin{aligned}
\limsup_{r\to 0}\mathcal{H}^{d-1}& \big((\{\operatorname{dist}(x,A) = r\}\cap \Omega') \setminus \{\operatorname{dist}(x,M)<3r\}\big)\\ &\leq \sum_{i\in I} \limsup_{r\to 0}\mathcal{H}^{d-1}\big((\{\operatorname{dist}(x,A) = r\}\cap (1+\theta)Q_i) \setminus \{\operatorname{dist}(x,M)<3r\}\big) \\
& \leq \sum_{i\in I} \mathcal{H}^{d-1}((\partial A \cap (1+\theta)Q_i) \setminus M)\\
& \leq \sum_{i\in I}\left(\mathcal{H}^{d-1}((\partial A \cap Q_i) \setminus M) +4(d-1) \theta  s_i^{d-1}\right)\\
& \leq \mathcal{H}^{d-1}(\partial A   \setminus M) + 4(d-1)\theta \mathcal{H}^{d-1}(\partial A) .
\end{aligned}
\end{equation}
As $0 < \theta \ll 1$ is arbitrary, this concludes the proof assuming that \eqref{eqn:jumpExclusion} holds on cubes. 

\emph{Step 3 (Cubes).} We prove the \textbf{claim} of Step 2. 
We first show that the levels-sets of the distance function are well-behaved in the cube, and then we use this to show that we can exclude the contribution of $M$ in the limit of the surface measure. Without loss of generality, we suppose that $Q = Q' \times (-s/2,s/2) = (-s/2,s/2)^d$ with $0\in \partial A$ and $s>0$, and we let $g_0:(1+\theta)Q'\to (-s/2,s/2)$ be the $C^1$ function having Lipschitz constant less than $\frac12$ such that $A$ is the subgraph of $g_0$ in $(1+\theta)Q$. Throughout we rely on the convention that $x = (x',x_d)\in \R^{d-1}\times \R$.

\emph{Substep 3.1 (Convergence of the level-sets).} 
We note that $\{\operatorname{dist}(x,A)=r\}\cap Q$ is the graph of a Lipschitz function $g_r:Q'\to (-s/2,s/2)$, for small $r>0,$ as the level-set of $\operatorname{dist}$ can be found by taking the supremum over all $r$-size translations of the graph of $g_0$, that is, 
\begin{equation}\label{eqn:ghjkt}
g_r(x') = \sup_{(\nu',\nu_d) \in \mathbb{S}^{d-1}}\{g_0(x'+r\nu')- r\nu_d\}.
\end{equation} 
Since $g_0$ is $C^1$, the graph actually converges in $W^{1,\infty}(\Omega)$ with 
\begin{equation}\label{eqn:infinityConvergence}
\|\nabla g_r - \nabla g_0\|_{L^\infty (Q')} \leq \omega(r),
\end{equation}
where $\omega$ is the (continuous) modulus of continuity of $\nabla g_0$ in $(1+\theta)Q'$. To see this, we note that for sufficiently small $h\in \R^{d-1}$ and $x' \in Q'$, by \eqref{eqn:ghjkt}, there is always a $\nu_{h}\in \mathbb{S}^{d-1}$ such that $g_r(x'+h) = g_0(x'+h+r\nu'_{h})- r\nu_{h,d}$. Using the mean value theorem, we can estimate
\begin{align*}
g_r(x'+h) - g_r(x') &\leq  g_0(x'+h+r\nu'_{h})- r\nu_{h,d} -(g_0(x'+r\nu'_{h})- r\nu_{h,d}) \\
&=  \langle\nabla g_0 (x' +\vartheta h + r\nu'_{h}), h\rangle \leq   \langle\nabla g_0 (x'), h\rangle + |h|\omega(|h| + r),
\end{align*}
where $\vartheta \in (0,1);$ one obtains the analogous bound from below using $\nu_0$ (defined taking $h=0$). Now assuming that $x'$ is a point of differentiability for $g_r,$ we insert $g_r(x'+h) - g_r(x') = \langle\nabla g_r(x'),h\rangle + o(|h|)$ into the above inequality to find
$$|\langle\nabla g_r(x') - \nabla g_0 (x'),h\rangle|\leq o(|h|) + |h|\omega(|h|+r). $$
Fixing $\tau \in (0,1)$, taking the supremum over $h \in \partial B(0,\tau)$, dividing by $\tau$, and then sending $\tau\to 0$, this inequality becomes
$$|\nabla g_r(x') - \nabla g_0 (x')|\leq \omega(r).$$ As Lipschitz functions are differentiable almost everywhere, we recover \eqref{eqn:infinityConvergence}.

\emph{Substep 3.2 (Exclusion of $M$).}
Let $\pi(x) = x'$ be the projection onto the first $(d-1)$-coordinates. We show that 
\begin{equation}\label{eqn:projectSubset}
\pi\big((\{\operatorname{dist}(x,A)=r\}\cap Q) \setminus \{\operatorname{dist}(x,M) <3r\}\big) \subset \pi((\partial A\cap Q) \setminus M).
\end{equation}
Supposing momentarily that the above set-relation is verified, we use this and the area formula to find
\begin{equation}\nonumber
\begin{aligned}
\mathcal{H}^{d-1}\big((\{\operatorname{dist}(x,A)=r\}\cap Q) \setminus \{\operatorname{dist}(x,M) <3r\}\big) &\leq \int_{\pi((\partial A\cap Q) \setminus M)} \sqrt{1+|\nabla g_r|^2}\, d x'
\end{aligned}
\end{equation}
Taking the $\limsup$ of both sides, using \eqref{eqn:infinityConvergence} and the area formula once again, we have
\begin{equation}\nonumber
\begin{aligned}
\limsup_{r \to 0}\mathcal{H}^{d-1}\big((\{\operatorname{dist}(x,A)=r\}\cap Q) \setminus \{\operatorname{dist}(x,M) <3r\}\big) &\leq   \mathcal{H}^{d-1}((\partial A\cap Q) \setminus M),
\end{aligned}
\end{equation}
proving the \textbf{claim} of Step 2. It remains to prove \eqref{eqn:projectSubset}.

Let $x\in (\{\operatorname{dist}(\tilde x,A)=r\}\cap Q) \setminus \{\operatorname{dist}(\tilde x,M) <3r\}$ and $y \in M \cap \partial A$. We prove $|x' - y'|\geq r$ from which \eqref{eqn:projectSubset} immediately follows. Letting $z\in \partial A$ be such that $|x-z| = \operatorname{dist}(x,A) = r$ and noting by choice of $x$ that $|x-y|\geq 3r$, we estimate
\begin{equation}\nonumber
|x' -y'| \geq  3r- |x_d - y_d|
\geq  3r -  r - \tfrac12 |z' - y'| 
\geq  3r -\tfrac{3}{2}r - \tfrac12 |x' - y'|,
\end{equation}
where we have used the triangle inequality twice and that $|z_d - y_d|\leq \tfrac12 |z'-y'|$ by the Lipschitz continuity of $g_0.$
Rearranging, we have
$|x' -y'| \geq   r $
as desired.
\end{proof}

We make a brief remark on how to construct near optimal profiles for Cahn--Hilliard energies in dimension $d=1$, which will be necessary in the proof of the $\limsup$ inequality.

\begin{remark}[Near optimal transitions]\label{rmk:nearOptimal}
{\normalfont
These ideas were introduced in \cite{Modica87}, see also \cite[Lemma 4.5]{cristoferiGravina}. For a fixed parameter $\lambda>0$, $\eps > 0$, and potential $f: \R\to [0,\infty)$, we define $\zeta_\eps^f$ as
\begin{equation}\nonumber
\zeta_\eps^f(s) : = \int_0^s \frac{\eps}{\sqrt{\lambda + f(t)}}\, dt.
\end{equation}
The near optimal transition is then given by 
\begin{equation}\nonumber
g_\eps^f(r):= \begin{cases}
0 & \text{ if }r<0,\\
(\zeta_\eps^f)^{-1}(r) & \text{ if } r\in [0,\zeta_\eps^f(1)], \\
1 & \text{ if } r>\zeta_\eps^f(1).
\end{cases}
\end{equation}
 Using that 
 \begin{equation}\label{eqn:derivativeOptTrans}
 \tfrac{d}{dr}g_\eps^f(r) = \frac{\sqrt{\lambda + f(g_\eps^f(r))}}{\eps} \quad \text{ for }r\in (0,\zeta_\eps^f(1)),
 \end{equation} we have
 \begin{align}
  \int_0^{\zeta_\eps^f(1)}\left(\frac{1}{\eps}f(g_\eps^f(r)) + \eps\left|\tfrac{d}{dr}g_\eps^f(r)\right|^2 \right) dr \leq& \int_0^{\zeta_\eps^f(1)} 2\sqrt{\lambda + f(g_\eps^f(r))}\left|\tfrac{d}{dr}g_\eps^f(r)\right| \, dr \nonumber \\
 = & \int_0^1 2 \sqrt{\lambda + f(r)} \, dr \leq \int_0^1 2 \sqrt{f(r)} \, dr + 2\sqrt{\lambda} ,\label{eqn:nearOptimal}
\end{align}  
where the first equality was found by a change of variables.
Finally, we note that $\zeta_\eps^f(1) \leq \eps/\sqrt{\lambda}$, which shows that the near optimal energy is found on a diffuse interface of width $O(\epsilon).$
}
\end{remark}

We now turn to the proof of the $\limsup$ inequality.

\begin{proof}[Proof of the $\limsup$ inequality in Theorem \ref{thm:main}.] Let $(c,u)\in BV(\Omega;\{0,1\})\times GSBD^2(\Omega)$ with ${E[c,u]<\infty}$. We construct a recovery sequence $(c_\eps,u_\eps,z_\eps)$ for such a pair. We first show, by a density argument, that we may suppose $c$ and $u$ have regular jump-sets. With regular jump-sets, our recovery sequence consists of the expected phase transitions and diffuse fracture approximations, with some care required to account for the overlap of the phase boundary and crack, or the irregular set in the approximation of Lemma \ref{lem:setApprox}. As before, we use the convention that $\delta = \delta(\eps).$

\emph{Step 1 (Application of density).} Fix $\eta>0.$
By the result of \cite[Theorem 1.1]{chambolleCrismale_approx} for any $u\in GSBD^2(\Omega)$, there is a set $M\subset \Omega$ that is a closed subset of a finite union of $C^1$ manifolds and a function $v \in SBV^2(\Omega;\R^d)\cap W^{1,\infty}(\Omega\setminus M;\R^d)$ such that 
$$\mathcal{H}^{d-1}(J_u\triangle J_v) <\eta, \quad \quad \mathcal{H}^{d-1}(J_v\triangle M) = 0,\quad \quad \|e(u)-e(v)\|_{L^2(\Omega)}<\eta, \quad \text{ and } \quad  d_{\rm meas}(u,v)<\eta,$$
 where $d_{\rm meas}$ is a metric for convergence in measure. 
Likewise, we apply Lemma \ref{lem:setApprox} to find $A \subset \Omega$ with $\partial A$ having regularity as specified in the lemma and $\|c - \chi_A\|_{BV(\Omega)}<\eta$---{of course we can also assume $\|c - \chi_A\|_{L^2(\Omega)}<\eta$}. 
It follows from a direct estimate that
$E[\chi_A,v] \leq E[c,u]+{C(c,u)}\eta $
and $\|c-\chi_A\|_{L^1(\Omega)} + d_{\rm meas}(u,v)\leq C\eta.$ By a diagonalization argument for $\eta\to 0$, it suffices to construct a recovery sequence for the pair $(\chi_A,v)$, which from now on we still denote by $(c,u)$. 

\emph{Step 2 (Construction of recovery sequences).}
We emphasize that $E[c,u]<\infty$, $c =: \chi_A$ is regular as in the approximation of Lemma \ref{lem:setApprox} with $\partial A$ a $C^1$ manifold (in a neighborhood of $\Omega$) away from $N\subset \subset \Omega$, a closed $\mathcal{H}^{d-1}$-null-set, and $u$ belongs to $W^{1,\infty}(\Omega\setminus M;\R^d)$ where $J_u$ is coincident with the set $M$ up to an $\mathcal{H}^{d-1}$-null-set. Importantly, the Minkowski content of $M$ exists with 
\begin{equation}\label{eqn:MMinkowski}
\lim_{r \to 0}\frac{\mathcal{L}^d(\dist(x,M)<r)}{2r} = \mathcal{H}^{d-1}(M)
\end{equation} by \cite[Theorem 2.106]{AmbrosioFuscoPallara}.

We fix $0<\lambda<1$ and recall the near optimal profiles $g^W_\eps$ and $g^V_{\delta}$ constructed in Remark \ref{rmk:nearOptimal}. We also rely on a smooth truncation function $\psi:\R\to [0,1]$ such that $\psi(r) = 0$ for $r<0$ and $\psi(r) = 1$ for $r>1.$ Since $\mathcal{H}^{d-1}(N) = 0$, we may cover $N$ with an open set $\mathcal{U}$ with $\mathcal{L}^d(\mathcal{U})<\eta$ such that $\partial \mathcal{U}\subset \subset \Omega$ is a $C^1$ manifold without boundary, $\mathcal{H}^{d-1}(\partial \mathcal{U}) < \eta$, and lastly $\mathcal{H}^{d-1}(\partial \mathcal{U} \cap \partial A) = 0$. To construct this set, one can cover $N$ with finitely many cubes compactly contained in $\Omega$ with the sum of their boundaries less than $\eta$,  mollify the characteristic function associated to the set, and then use Sard's theorem and the area formula to select a super level-set of the mollified function having the desired properties.

With these objects in hand, we can define the recovery sequence: The phase transition is given by
\begin{equation}\label{eqn:cRecovDef}
c_\eps(x) : = g_\eps^W(\operatorname{dist}(x,A))\psi\left(\frac{\operatorname{dist}(x,\mathcal{U})}{\eps}\right),
\end{equation}
where $\operatorname{dist}(x,A)$ is the distance to the extended set $A\subset \R^d.$
Recalling from Remark \ref{rmk:nearOptimal} that $g_\eps^W$ transitions between $0$ and $1$ on an interval of width at most $\eps/\sqrt{\lambda}\leq {\lambda}\delta$, where the inequality follows by the assumption that $\eps/\delta \to 0$ as $\eps\to 0$, we define 
\begin{equation}\nonumber
z_\eps(x) : = g_\delta^V\left(\operatorname{dist}(x,M) - {\lambda}\delta\right)
\end{equation}
so that $\{c_\eps \neq 0 \text{ or }1\}\subset \{z_\eps = 0\}\cup \{\operatorname{dist}(x,\mathcal{U})<\eps\}$ near $M \cap \partial A.$
Finally, we define the $C^1$ function
\begin{equation}\nonumber
u_\eps (x) : = u(x) \psi\left(\frac{\operatorname{dist}(x,M)}{\lambda\delta}\right),
\end{equation}
so that $\{u_\eps \neq u\}\subset \{z_\eps = 0\}$. It is straightforward to show that $(c_\eps,u_\eps,z_\eps)$ converges to $(c,u,1)$ in measure.

\emph{Step 3 (Limit energy).}
We go through each term of the energy individually.

\emph{Substep 3.1 (Elastic energy).}
Note, since $u\in W^{1,\infty}(\Omega\setminus M)$, that $|\nabla u_\eps(x)|\leq C(u)/(\lambda\delta)$ for $x$ with $\dist(x,M)<\lambda\delta$ and ${|\nabla u_\eps(x)|\leq C(u)}$ otherwise.
Consequently, defining $\mathcal{U}_\eps : = \{\dist(x,\mathcal{U})<\eps\}$, we can estimate that 
\begin{align*}
\int_{\Omega} & (z_\eps^2 + \delta^2)\CC{(e(u_\eps) - c_\eps e_0)}\, dx \\
\leq & \int_{\{\dist(x,M)>{\lambda}\delta \}\cap\left(\{\dist(x, \partial A)>\eps/\sqrt{\lambda}\}\cap \, \mathcal{U}_\eps^c \right)}(1 + \delta^2)\CC{(e(u) - c e_0)}\, dx \\
& +\int_{\{\dist(x,M)<{\lambda}\delta\}}C\delta^2 |e(u_\eps) - c_\eps e_0|^2 \, dx + \int_{\{\dist(x,M)>{\lambda}\delta \}\cap \left( \{ \dist(x, \partial A)<\eps/\sqrt{\lambda}\}\cup \, \mathcal{U}_\eps \right) }C|e(u_\eps) - c_\eps e_0|^2\, dx \\
\leq & (1+\delta^2)\int_{\Omega}  \CC{(e(u) - c e_0)}\, dx \\
& +C(u)\left(\frac{\mathcal{L}^d(\{\dist(x,M)<\lambda\delta\})}{{\lambda}\delta} \frac{\delta^2}{{\lambda}\delta} +  \mathcal{L}^d(\{\dist(x, \partial A)<\eps/\sqrt{\lambda}\}) +\mathcal{L}^d(\mathcal{U}_\eps)\right).
\end{align*}
By \eqref{eqn:MMinkowski} {and the regularity of $\partial A$ and $\partial \mathcal{U}$,} the last three terms are controlled by $\eta$ as $\eps \to 0$, and we have
\begin{equation}\label{eqn:elasticUpperBound}
\limsup_{\eps\to 0}\int_{\Omega}  (z_\eps^2 + \delta^2)\CC{(e(u_\eps) - c_\eps e_0)}\, dx \leq \int_{\Omega}  \CC{(e(u) - c e_0)}\, dx +C\eta.
\end{equation}

\emph{Substep 3.2 ({Interfacial} energy).} 
We recover the correct phase energy by applying the coarea formula in $\Omega\setminus \mathcal{U}$ with respect to $\dist(\cdot, A)$, wherein the hypotheses of Lemma \ref{lem:jumpExclusion} are satisfied for $\partial A$ in $\Omega\setminus \mathcal{U}$ and $\dist(\cdot, A)$ is defined in terms of $A\subset \R^d$. Precisely, noting that $\|\nabla c_\eps\|_{L^\infty(\Omega)}\leq C/\eps$ due to \eqref{eqn:cRecovDef} and \eqref{eqn:derivativeOptTrans}, we have 
\begin{align}
& \int_{\Omega}  \phi_\delta(z_\eps) \left(\frac{1}{\eps}W(c_\eps) + \eps \|\nabla c_\eps\|^2\right) dx \nonumber \\
&\leq  \int_{\{\dist(x,M)>{\lambda}\delta\}\cap \{\dist(x,\mathcal{U})>\eps\}}\phi_\delta(1)\left(\frac{1}{\eps}W(c_\eps) + \eps |\nabla c_\eps|^2\right) dx \nonumber\\
& \quad + \int_{\{\dist(x,M)<{\lambda}\delta\}\cap \{\dist(x,\mathcal{U})>\eps\}} \phi_\delta(0)\left(\frac{1}{\eps}W(c_\eps) + \eps |\nabla c_\eps|^2\right) dx + \int_{ \{0<\dist(x,\mathcal{U})<\eps\}}\frac{C}{\eps}\, dx \nonumber \\
&\leq  \phi_\delta (1)\int_{0}^{\tfrac{\eps }{\sqrt{\lambda}}} \left(\frac{1}{\eps}W(g_\eps^W(r)) + \eps |\tfrac{d}{dr}g_\eps^W(r)|^2\right)\mathcal{H}^{d-1}\left((\{\dist(x,A) = r\} \cap (\Omega\setminus \mathcal{U}))\setminus \{\dist(x,M) >{\lambda}\delta\} \right)  dr  \nonumber \\
& \quad + \phi_\delta (0)\int_{0}^{\tfrac{\eps }{\sqrt{\lambda}}} \left(\frac{1}{\eps}W(g_\eps^W(r)) + \eps |\tfrac{d}{dr}g_\eps^W(r)|^2\right)\mathcal{H}^{d-1}\left(\{\dist(x,A) = r\} \cap (\Omega\setminus \mathcal{U})\right)  dr \nonumber\\
& \quad + C\frac{\mathcal{L}^d(\{0<\dist(x,\mathcal{U})<\eps\})}{\eps} . \label{eqn:phaseEnergy}
\end{align}
Using \eqref{eqn:nearOptimal} and recalling the definition of $\alpha_{\rm surf}$ in \eqref{eqn:energyDensities}, we can bound the third-to-last term in the above inequality by
\begin{equation}\label{eqn:helpfulEst}
\phi_\delta (1)\left( \alpha_{\rm surf} + 2\sqrt{\lambda}\right)\sup_{0<r<\eps/\sqrt{\lambda}}\left\{\mathcal{H}^{d-1}((\{\dist(x,A) = r\} \cap (\Omega\setminus \mathcal{U}))\setminus \{\dist(x,M) >{\lambda}\delta\} )\right\}.
\end{equation}
A bound analogous to \eqref{eqn:helpfulEst} holds for the second-to-last term of \eqref{eqn:phaseEnergy}, instead with $\phi_\delta(1)$ and $M$ replaced by $\phi_\delta(0)$ and the empty-set $\emptyset$, respectively.
As ${\lambda}\delta > 3\eps/\sqrt{\lambda}$ for sufficiently small $\eps$, we can apply the bound \eqref{eqn:helpfulEst} and Lemma \ref{lem:jumpExclusion} (recall $\phi_\delta(1)\to 1$ and $\phi_\delta(0)\to 0$) while taking $\eps\to 0$ in \eqref{eqn:phaseEnergy} to find
\begin{equation}\label{eqn:phaseUpperBound}
\limsup_{\eps\to 0} \int_{\Omega}  \phi_\delta(z_\eps) \left(\frac{1}{\eps}W(c_\eps) + \eps \|\nabla c_\eps\|^2\right) dx \leq \alpha_{\rm surf}\mathcal{H}^{d-1}((\partial A \cap (\Omega \setminus \mathcal{U})) \setminus M) + C(\eta+\sqrt{\lambda}),
\end{equation}
where $C = C(A)>0$ and we have also used that $\lim_{\eps\to 0}\frac{\mathcal{L}^d(\{0<\dist(x,\mathcal{U})<\eps\})}{\eps} = \mathcal{H}^{d-1}(\partial \mathcal{U})<\eta$ (this follows from a one-sided version of \cite[Theorem 2.106]{AmbrosioFuscoPallara}, but an upper bound suffices).

\emph{Substep 3.3 (Crack energy).} We argue using an integration by parts trick, as in \cite{ambrosio-tortorelli-1990}.
The diffuse crack energy can be split into two regions as
\begin{align}
&\int_{\Omega}\left(\frac{1}{\delta}V(z_\eps) + \delta |\nabla z_\eps|^2\right)dx \nonumber \\
& = \int_{\{\dist(x,M)<\lambda\delta\}}\frac{1}{\delta}V(0)\, dx + \int_{\{\lambda\delta<\dist(x,M)<\lambda\delta+\zeta_\delta^V(1)\}}\left(\frac{1}{\delta}V(z_\eps) + \delta |\nabla z_\eps|^2\right)dx \nonumber \\
& = \lambda V(0)\frac{\mathcal{L}^d(\{\dist(x,M)<\lambda\delta\})}{\lambda\delta} \nonumber \\
& \quad +\int_{0}^{\zeta_\delta^V(1)}\left(\frac{1}{\delta}V(g_\delta^V(r )) + \delta |\tfrac{d}{dr}g_\delta^V(r )|^2\right)
\mathcal{H}^{d-1}(\{\dist(x,M) = r+\lambda \delta \})\, dr,\label{eqn:crackSplit}
\end{align}
where in the last equality we used the coarea formula with respect to $x\mapsto \dist(x,M) - \lambda \delta$. We deal with the last term in the above display. For this, we define $$h_\delta(r) : =  \left(\frac{1}{\delta}V(g_\delta^V(r )) + \delta |\tfrac{d}{dr}g_\delta^V(r )|^2\right) \quad \text{ for }r \in [0,\zeta_\delta^V(1)],$$
and note that $h_\delta \in C^1([0,\zeta_\delta^V(1)])$ since $(\zeta_\delta^V)^{-1}$ is in $C^2(\R)$, see Remark \ref{rmk:nearOptimal}. Note from \eqref{eqn:derivativeOptTrans} and $\tfrac{d^2}{dr^2}g_\delta^V(r) = V'(g_\delta^V(r))/(2\delta^2)$ that 
\begin{equation}\label{eqn:Hproperties}
h_\delta( 0 ) = \frac{2V(0) +\lambda}{\delta}, \quad  h_\delta(\zeta_\delta^V(1)) = \frac{\lambda}{\delta}, \quad \text{ and } \quad \tfrac{d}{dr} h_\delta(r) = \frac{2}{\delta}V'(g_\delta^V(r))\tfrac{d}{dr}g_\delta^V(r)\leq 0
\end{equation}
since $V' \leq 0$ by the assumption above \eqref{ass:V0}.
Further, we define the absolutely continuous function $$\operatorname{Vol}(r) : = \mathcal{L}^d(\{\dist(x,M)<r\}) \quad \text{ for } r>0$$ for which $\tfrac{d}{dr}\operatorname{Vol}(r) = \mathcal{H}^{d-1}(\{\dist(x,M)=r\})$ holds for almost every $r>0$---this can be shown using the coarea formula for $\mathcal{L}^d( \{s\leq\dist(x,M)<r\})$.
Consequently, an integration by parts allows us to rewrite the last term of \eqref{eqn:crackSplit} as
\begin{equation}\label{eqn:relation1}
\int_{0}^{\zeta_\delta^V(1)} h_\delta(r) \tfrac{d}{dr}\operatorname{Vol}(r+\lambda \delta)\, dr = [ h_\delta(r) \operatorname{Vol}(r+\lambda \delta) |_{r=0}^{\zeta_\delta^V(1)} - \int_{0}^{\zeta_\delta^V(1)} \tfrac{d}{dr} h_\delta(r) \operatorname{Vol}(r+\lambda \delta)\, dr.
\end{equation}
To estimate the last term above, by \eqref{eqn:MMinkowski}, for $\sigma>0$, we have that
\begin{equation}\label{eqn:minkowskiM}
\left|\frac{\operatorname{Vol}(r)}{2r} - \mathcal{H}^{d-1}(M)  \right|<\sigma \quad \text{ for all } 0 <r \ll 1,
\end{equation}
so that for $\delta \ll 1$, we have
\begin{align}
- \int_{0}^{\zeta_\delta^V(1)} \tfrac{d}{dr} h_\delta(r) \operatorname{Vol}(r+\lambda \delta)\, dr &\leq -2(\mathcal{H}^{d-1}(M) +\sigma)  \int_{0}^{\zeta_\delta^V(1)} \tfrac{d}{dr} h_\delta(r) (r+\lambda \delta) \, dr \nonumber \\
& \leq 2(\mathcal{H}^{d-1}(M) +\sigma) \left(\int_{0}^{\zeta_\delta^V(1)}h_\delta(r)\, dr - [h_\delta(r) (r+\lambda \delta)|_{r=0}^{\zeta_\delta^V(1)}\right),\label{eqn:relation2}
\end{align}
where we have used the last relation in \eqref{eqn:Hproperties} and an integration by parts. Using \eqref{eqn:nearOptimal} while recalling that $\alpha_{\rm frac}  := 4\int_0^1\sqrt{V(s)}ds$, evaluating Dirac masses (for lack of a simpler phrase) at the boundary points and dropping the negative terms, we conclude from \eqref{eqn:relation1} and \eqref{eqn:relation2} that
\begin{align}
& \int_{0}^{\zeta_\delta^V(1)} h_\delta(r) \tfrac{d}{dr}\operatorname{Vol}(r+\lambda \delta)\, dr \nonumber \\
& \leq (\mathcal{H}^{d-1}(M) +\sigma) (\alpha_{\rm frac} + 4\sqrt{\lambda}) + C(M)\left(\frac{\lambda}{\delta}\operatorname{Vol}(\zeta_\delta^V(1)+\lambda\delta) + \frac{2V(0) +\lambda}{\delta}(\lambda \delta) \right)  \nonumber \\
& \leq \alpha_{\rm frac}\mathcal{H}^{d-1}(M) +C(M)(\sigma +\sqrt{\lambda} ) \label{eqn:crackOuterBound}
\end{align}
where in the last line we have used $\zeta_\delta^V(1)\leq \delta /\sqrt{\lambda}$ and \eqref{eqn:minkowskiM} with $r =2\tfrac{\delta }{\sqrt{\lambda}} \geq \tfrac{\delta }{\sqrt{\lambda}} + \lambda\delta$ (we remark that the above estimate is why $\lambda$ was included in the definition of $z_\eps$).
Putting together \eqref{eqn:crackSplit} and \eqref{eqn:crackOuterBound}, while applying \eqref{eqn:minkowskiM} once again for $\delta(\eps) \ll 1$, we have
\begin{equation}\nonumber
\int_{\Omega}\left(\frac{1}{\delta}V(z_\eps) + \delta |\nabla z_\eps|^2\right)dx \leq  \alpha_{\rm frac}\mathcal{H}^{d-1}(M) +C(M)(\sigma + \sqrt{\lambda}  ).
\end{equation}
Taking the $\limsup$ of both sides, we make take $\sigma \to 0$ to find
\begin{equation}\label{eqn:crackUpperBound}
\limsup_{\eps \to 0}\int_{\Omega}\left(\frac{1}{\delta}V(z_\eps) + \delta |\nabla z_\eps|^2\right)dx \leq  \alpha_{\rm frac}\mathcal{H}^{d-1}(M) +C(M)\sqrt{\lambda} .
\end{equation}

\emph{Step 4 (Conclusion).} With \eqref{eqn:elasticUpperBound}, \eqref{eqn:phaseUpperBound}, and \eqref{eqn:crackUpperBound}, we have
$$\limsup_{\eps\to 0}E_\eps[c_\eps,u_\eps,z_\eps]\leq E[c,u] + C(\eta +\sqrt{\lambda}),$$
from which a diagonalization as $\eta\to 0$ and $\lambda\to 0$ concludes the construction of the recovery sequence.
\end{proof}

\section*{Acknowledgements}
This work grew out of S.W.'s Master's thesis at the Unviersity of Bonn. During much of the writing of the paper K.S. was at the Institute for Applied Mathematics at the University of Bonn. K.S. was supported by funding from the Deutsche Forschungsgemeinschaft (DFG, German Research Foundation) 
under Germany's Excellence Strategy -- EXC-2047/1 -- 390685813, the DFG project 211504053 - SFB 1060; also by the NSF (USA) under awards DMS-2108784 and DMS-2136198.

\bibliographystyle{amsplain}
\bibliography{SW_arxiv.bib}

\providecommand{\bysame}{\leavevmode\hbox to3em{\hrulefill}\thinspace}
\providecommand{\MR}{\relax\ifhmode\unskip\space\fi MR }
\providecommand{\MRhref}[2]{%
  \href{http://www.ams.org/mathscinet-getitem?mr=#1}{#2}
}
\providecommand{\href}[2]{#2}
\begin{thebibliography}{10}

\bibitem{AlmiTasso_2023}
S.~Almi and E.~Tasso, \emph{A new proof of compactness in {G(S)BD}}, Advances
  in Calculus of Variations \textbf{16} (2023), no.~3, 637--650.

\bibitem{AmbrosioFuscoPallara}
L.~Ambrosio, N.~Fusco, and D.~Pallara, \emph{Functions of bounded variation and
  free discontinuity problems}, Oxford Mathematical Monographs, The Clarendon
  Press, Oxford University Press, New York, 2000.

\bibitem{ambrosio-tortorelli-1990}
L.~Ambrosio and V.~M. Tortorelli, \emph{Approximation of functionals depending
  on jumps by elliptic functionals via {$\Gamma$}-convergence}, Comm. Pure
  Appl. Math. \textbf{43} (1990), no.~8, 999--1036. \MR{1075076}

\bibitem{Bazant-Theory2013}
M.~Z. Bazant, \emph{Theory of chemical kinetics and charge transfer based on
  nonequilibrium thermodynamics}, Accounts of chemical research \textbf{46}
  (2013).

\bibitem{braidesFreeDiscont}
A.~Braides, \emph{Approximation of free discontinuity problems}, Springer,
  1998.

\bibitem{Brad02}
\bysame, \emph{Gamma-convergence for beginners}, vol.~22, Oxford University
  Press, Oxford, 2002.

\bibitem{brescianiFriedrichMoraCorral}
M.~Bresciani, M.~Friedrich, and C.~{Mora-Corral}, \emph{Variational models with
  {Eulerian}-{Lagrangian} formulation allowing for material failure}, arXiv
  preprint (2024), \href{https://arxiv.org/abs/2402.12870}{arXiv:2402.12870}.

\bibitem{bungertStinson2024}
L.~Bungert and K.~Stinson, \emph{{$\Gamma$-convergence} of a nonlocal perimeter
  arising in adversarial machine learning}, Calc. Var. Partial Differential
  Equations \textbf{63} (2024), no.~114, 1--39.

\bibitem{chambolleCrismale18}
A.~Chambolle and V.~Crismale, \emph{Compactness and lower semicontinuity in
  {GSBD}}, J. Eur. Math. Soc. \textbf{23} (2018), no.~3, 701--719.

\bibitem{chambolleCrismale_approx}
A.~Chambolle and V.~Crismale, \emph{A density result in {$GSBD^p$} with
  applications to the approximation of brittle fracture energies}, Arch.
  Rational Mech. Anal. \textbf{232} (2019), 1329--1378.

\bibitem{chambolleCrismale-Hetero}
\bysame, \emph{A general compactness theorem in {$G(S)BD$}}, To appear in:
  Indiana Univ. Math. J. (2022),
  \href{https://arxiv.org/abs/2210.04355}{arXiv:2210.04355}.

\bibitem{chambolleCrismaleEquilibrium}
\bysame, \emph{Equilibrium configurations for nonhomogeneous linearly elastic
  materials with surface discontinuities}, Annali della Scuola Normale
  Superiore di Pisa - Classe di Scienze \textbf{XXIV} (2023), no.~5,
  1575--1610.

\bibitem{conti-SO2rigid}
S.~Conti and B.~Schweizer, \emph{Rigidity and gamma convergence for
  solid‐solid phase transitions with {SO(2)} invariance}, Comm. Pure Appl.
  Math. \textbf{59} (2006), 830 -- 868.

\bibitem{cristoferiGravina}
R.~Cristoferi and G.~Gravina, \emph{Sharp interface limit of a multi-phase
  transitions model under nonisothermal conditions}, Calc. Var. Partial
  Differential Equations \textbf{30} (2021), no.~142, 1--62.

\bibitem{Dal2015-Comp}
H.~Dal and C.~Miehe, \emph{Computational electro-chemo-mechanics of lithium-ion
  battery electrodes at finite strains}, Comput. Mech. \textbf{55} (2015),
  no.~2, 303--325.

\bibitem{DalMasoBook}
G.~Dal~Maso, \emph{An introduction to {$\Gamma$}-convergence}, Progress in
  Nonlinear Differential Equations and their Applications, 8, Birkh\"auser
  Boston, Inc., Boston, MA, 1993. \MR{1201152 (94a:49001)}

\bibitem{dalMasoGSBD}
G.~{Dal Maso}, \emph{Generalised functions of bounded deformation}, J. Eur.
  Math. Soc. \textbf{15} (2013), no.~5, 1943--1997.

\bibitem{dePhilippis2017}
G.~{De Philippis}, N.~Fusco, and A.~Pratelli, \emph{On the approximation of
  {SBV} functions}, Atti Accad. Naz. Lincei Cl. Sci. Fis. Mat. Natur.
  \textbf{28} (2017), no.~2, 369--413.

\bibitem{fonseca2007modern}
I.~Fonseca and G.~Leoni, \emph{Modern methods in the calculus of variations:
  {$L^p$} spaces}, Springer Science \& Business Media, 2007.

\bibitem{friedrich2018-piecewiseKorn}
M.~Friedrich, \emph{A piecewise {K}orn inequality in {SBD} and applications to
  embedding and density results}, SIAM J. Math. Anal. \textbf{50} (2018),
  no.~4, 3842--3918.

\bibitem{friedrich2020}
M.~Friedrich, \emph{{G}riffith energies as small strain limit of nonlinear
  models for nonsimple brittle materials}, Mathematics in Engineering
  \textbf{2} (2020), no.~1, 75--100.

\bibitem{FriedrichSolombrino18}
M.~Friedrich and F.~Solombrino, \emph{Quasistatic crack growth in
  {$2d$}-linearized elasticity}, Ann. Inst. H. Poincar{\'e} Anal. Non
  Lin{\'e}aire \textbf{35} (2018), 27--64.

\bibitem{garcke-CHlogPot}
H.~Garcke, \emph{On a {Cahn–Hilliard} model for phase separation with elastic
  misfit}, Ann. Inst. H. Poincar{\'e} Anal. Non Lin{\'e}aire \textbf{22}
  (2005), no.~2, 165 -- 185.

\bibitem{HK_2013}
C.~Heinemann and C.~Kraus, \emph{Existence results for diffuse interface models
  describing phase separation and damage}, European J. Appl. Math. \textbf{24}
  (2013), no.~2, 179–211.

\bibitem{heinemannKraus14}
C.~Heinemann and C.~Kraus, \emph{Phase separation coupled with damage
  processes: Analysis of phase field models in elastic media}, Springer
  Spektrum Wiesbaden, 2014.

\bibitem{heinemannKraus_2015}
\bysame, \emph{A degenerating {Cahn--Hilliard} system coupled with complete
  damage processes}, Nonlinear Analysis: Real World Applications \textbf{22}
  (2015), 388--403.

\bibitem{HCRR_2017}
C.~Heinemann, C.~Kraus, E.~Rocca, and R.~Rossi, \emph{A temperature-dependent
  phase-field model for phase separation and damage}, Arch. Rational Mech.
  Anal. \textbf{225} (2017), 177--247.

\bibitem{henaoMoraCorral2010}
D.~Henao and C.~{Mora-Corral}, \emph{Invertibility and weak continuity of the
  determinant for the modelling of cavitation and fracture in nonlinear
  elasticity}, Arch. Rational Mech. Anal. \textbf{197} (2010), 619--655.

\bibitem{iurlano_density}
F.~Iurlano, \emph{A density result for {$GSBD$} and its application to the
  approximation of brittle fracture energies}, Calc. Var. Partial Differential
  Equations \textbf{51} (2014), 315--342.

\bibitem{kholmatovPiovano}
S.~Koholmatov and P.~Piovano, \emph{Existence of minimizers for the sdri model
  in ${\mathbb{r}^n}$: Wetting and dewetting regimes with mismatch strain},
  arXiv preprint (2023),
  \href{https://arxiv.org/abs/2305.10304}{arXiv:2305.10304}.

\bibitem{Li_2020}
P.~Li, Y.~Zhao, Y.~Shen, and {S.-H.} Bo, \emph{Fracture behavior in battery
  materials}, Journal of Physics: Energy \textbf{2} (2020), no.~2, 022002.

\bibitem{ModicaMortola}
L.~Modica and S.~Mortola, \emph{Un esempio di {$\Gamma$}-convergenza}, Boll.
  Un. Mat. Ital. B (5) \textbf{14} (1977), no.~1, 285--299.

\bibitem{Modica87}
Luciano Modica, \emph{The gradient theory of phase transitions and the minimal
  interface criterion}, Arch. Rational Mech. Anal. \textbf{98} (1987), no.~2,
  123--142. \MR{866718 (88f:76038)}

\bibitem{O_Connor_2016}
D.~T. O'Connor, M.~J. Welland, W.~K. Liu, and P.~W. Voorhees, \emph{Phase
  transformation and fracture in single {L}i$_x${FePO}$_4$ cathode particles: a
  phase-field approach to {L}i-ion intercalation and fracture}, Modelling and
  Simulation in Materials Science and Engineering \textbf{24} (2016), no.~3.

\bibitem{silhavy_2011}
M.~{\v{S}ilhav\'y}, \emph{Equilibrium of phases with interfacial energy: A
  variational approach}, J. Elast. \textbf{105} (2011), 271--303.

\bibitem{stinsonLiBatteryExpBC}
K.~Stinson, \emph{Existence for a {C}ahn--{H}illiard model for {L}ithium-ion
  batteries with exponential growth boundary conditions}, J. Nonlinear Sci.
  \textbf{33}, no.~69.

\bibitem{stinsonLiBatteryGamma}
\bysame, \emph{On {$\Gamma$-}convergence of a variational model for lithium-ion
  batteries}, Arch. Rational Mech. Anal. \textbf{240} (2021), 1--50.

\end{thebibliography}

\end{document}